\documentclass{article}

\usepackage{arxiv}

\usepackage[utf8]{inputenc} 
\usepackage[T1]{fontenc}    
\usepackage{hyperref}       
\usepackage{url}            
\usepackage{booktabs}       
\usepackage{amsfonts}       
\usepackage{nicefrac}       
\usepackage{microtype}      
\usepackage{lipsum}
\usepackage{graphicx}
\usepackage{amsmath,amssymb,amsthm,mathtools}
\usepackage{enumitem}
\usepackage{subcaption}
\usepackage{xcolor}
\newtheorem{theorem}{Theorem}[section]
\newtheorem{lemma}[theorem]{Lemma}
\newtheorem{proposition}[theorem]{Proposition}
\newtheorem{corollary}[theorem]{Corollary}
\newtheorem{definition}[theorem]{Definition}
\newtheorem{remark}[theorem]{Remark}
\newtheorem{example}[theorem]{Example}
\usepackage[numbers]{natbib}

\graphicspath{ {./images/} }

\title{Characterizations of bipartite and Eulerian partial duals of orientable hypermaps}

\author{
 Yufan Han \\
  College of Mathematical Sciences\\
  Xinjiang Normal University\\
  Urumqi 830017, China \\
  \texttt{} \\
   \And
 Metrose Metsidik \\
  College of Mathematical Sciences\\
  Xinjiang Normal University\\
  Urumqi 830017, China \\
  \texttt{metrose@xjnu.edu.cn} \\
}

\begin{document}
\maketitle
\begin{abstract}
For ribbon graphs, Deng et al. established two classical characterizations: a partial dual is bipartite if and only if the dualizing set consists of all $c$-edges induced by some all-crossing direction of the medial graph, and a partial dual is Eulerian if and only if the dualizing set consists of all $d$-edges induced by some crossing-total direction, possibly together with some $t$-edges. In this paper we extend both results to hypermaps in the $(\sigma,\alpha)$ model, restricting attention to partial duality with respect to subsets of hyperedges.

We first rewrite the Chmutov and Vignes-Tourneret's three-permutation formula as an explicit hyperedge-partial-duality formula in the two-permutation model, and show that in this model partial duality acts exactly by preserving the support and length of every hyperedge while reversing the $\alpha$-cycles corresponding to the selected hyperedges. Next, using the Cori and Hetyei's construction of the medial map, we define for each hyperedge subset $E'\subseteq E(H)$ a black/white smoothing state $S_{E'}$, and prove rigorously that the state circles of $S_{E'}$ are in bijection with the vertices of the partial dual $H^{E'}$. Consequently, $H^{E'}$ is Eulerian if and only if every state circle has even length.

On this basis we prove the following two main theorems:
\[
\begin{aligned}
H^{E'}\text{ is Eulerian}
&\Longleftrightarrow
\exists\text{ a crossing-total direction $\Omega$ of }M(H) \\
&\hspace{3.35em}\text{such that } E'=D(\Omega)\cup T',\quad T'\subseteq T(\Omega),
\end{aligned}
\]
\[
\begin{aligned}
H^{E'}\text{ is bipartite}
&\Longleftrightarrow
\exists\text{ an all-crossing direction $\Phi$ of }M(H) \\
&\hspace{3.35em}\text{such that } E'=C(\Phi).
\end{aligned}
\]
Here $D(\Omega)$, $T(\Omega)$ and $C(\Phi)$ denote, respectively, the sets of all $d$-type, $t$-type and $c$-type hyperedges. Unlike the ribbon-graph case, the hypermap setting exhibits a genuine new obstruction: if some hyperedge-partial dual is bipartite, then every hyperedge of the original hypermap must have even length.
\end{abstract}


\section{Introduction}

Let $G$ be a ribbon graph and let $A\subseteq E(G)$. In the planar setting, Huggett and Moffatt first extended the classical correspondence between Eulerian graphs and bipartite geometric duals to partial duality, and obtained a medial-graph characterization of bipartite partial duals of plane graphs. Metsidik and Jin then characterized Eulerian partial duals of plane graphs. Deng, Jin and Metsidik subsequently generalized both results to arbitrary ribbon graphs and expressed them uniformly in terms of all-crossing and crossing-total directions of medial graphs \cite{HuggettMoffatt2013,MetsidikJin2018,DengJinMetsidik2020}. These results show that global properties of partial duals can be completely controlled by local orientation data on the medial graph.

A hypermap can be viewed as a hypergraph embedded in a surface; when every hyperedge has length $2$, it reduces to a ribbon graph. Classical combinatorial models for hypermaps go back to Cori, Walsh, Jones, and Singerman, Lando, and Zvonkin \cite{Cori1975,Walsh1975,JonesSingerman1978,LandoZvonkin2004}. Chmutov and Vignes-Tourneret extended partial duality to general hypermaps and gave an explicit formula in the orientable three-permutation model \cite{ChmutovVignes2022}; Cori and Hetyei, in their study of the Whitney polynomial of a hypermap, systematically employed a permutation-based construction of the medial map \cite{CoriHetyei2025}. These works provide a natural technical foundation for the present paper.

We restrict throughout to orientable hypermaps and consider only partial duality with respect to subsets of hyperedges. In this setting, the orientation of the underlying surface yields a canonical checkerboard coloring of the medial map, so the $x/y$-labeling according to whether the black face lies on the left or on the right of an oriented edge is well defined. Our main contributions can be summarized as follows:
\begin{enumerate}
  \item We give a rigorous and unambiguous formula for hyperedge partial duality in the two-permutation $(\sigma,\alpha)$ model.
  \item We provide a complete proof of the correspondence ``state circles $\leftrightarrow$ vertices of a partial dual,'' which is the common core of both main theorems.
  \item We extend the crossing-total and all-crossing characterizations from ribbon graphs to the hypermap setting, where medial vertices may have arbitrary valence.
  \item We isolate a hypermap-specific even-length obstruction: if some hyperedge-partial dual is bipartite, then every hyperedge of the original hypermap must have even length.
\end{enumerate}

We outline the structure of this paper in detail. In Section 2, we start with fundamental notations of orientable hypermaps under the two-permutation model, transform the existing three-permutation partial duality formula into a new form adapted for hyperedges, and discuss elementary properties of Eulerian and bipartite hypermaps. Section 3 is devoted to the permutation definition of medial maps, their checkerboard colorings, and the construction of states and state circles from smoothing operations. The key bijection linking state circles to vertices of partial duals is rigorously proved in Section 4. Based on this correspondence, Section 5 proves the characterization theorem for Eulerian partial duals with respect to crossing-total directions, while Section 6 addresses bipartite partial duals and all-crossing directions. Section 7 demonstrates all theoretical conclusions through a concrete hypermap example, and compares our results with the known results on ribbon graphs. We also explain why non-orientable hypermaps are not considered in this work. The appendix contains additional figures visualizing Eulerian and bipartite partial duals of the example hypermap.

\section{Orientable hypermaps and hyperedge partial duality}

\subsection{The two-permutation \texorpdfstring{$(\sigma,\alpha)$}{(sigma, alpha)} model}

Throughout the paper, permutations are composed from right to left. We adopt the permutation model used by Cori and Hetyei in \cite[\S 1.1]{CoriHetyei2025}. Let $D=[n]=\{1,2,\dots,n\}$ be a finite set of darts. A connected orientable hypermap is a pair of permutations
\[
H=(\sigma,\alpha)\in S_n\times S_n,
\]
acting transitively on $D$. It encodes a hypergraph embedded in an orientable surface without boundary: the cycles of $\sigma$ correspond to vertices, the cycles of $\alpha$ correspond to hyperedges, and the faces correspond to the cycles of the permutation
\[
\phi:=\alpha^{-1}\sigma.
\]
The degree of a vertex or a hyperedge is defined as the length of the corresponding cycle. For $i\in D$, let $v(i)$ denote the $\sigma$-cycle containing $i$, and let $e(i)$ denote the $\alpha$-cycle containing $i$.

\begin{definition}\label{def:ve-vb}
Let $H=(\sigma,\alpha)$ be an orientable hypermap.
\begin{enumerate}
  \item Following Euler's classical terminology for graphs of even degree \cite{Euler1736}, we say that $H$ is \emph{Eulerian} if every cycle of $\sigma$ has even length.
  \item A hypermap is \emph{bipartite} if its vertices admit a 2-coloring such that adjacent vertices receive distinct colors. Equivalently, \(H\) is \emph{bipartite} iff its vertex-adjacency graph \(\Gamma(H)\)~\cite{BredaDuarte2007} is bipartite, where \(\Gamma(H)\) is the graph whose vertices are the vertices of \(H\), and for each dart \(i\in D\), an edge is added between the vertices \(v(i)\) and \(v(\alpha(i))\); loops and multiple edges are allowed.
\end{enumerate}
\end{definition}

For later reference, consider the orientable hypermap $H_0=(\sigma,\alpha)$ given by
\[
\sigma=(1\,8)(2\,10\,11)(3\,6)(4\,5\,7)(9\,12),
\qquad
\alpha=(1\,2\,3\,4)(5\,6\,7\,8)(9\,10)(11\,12).
\]
Write the four hyperedges as
\[
A:=(1,2,3,4),\qquad B:=(5,6,7,8),\qquad C:=(9,10),\qquad D:=(11,12).
\]

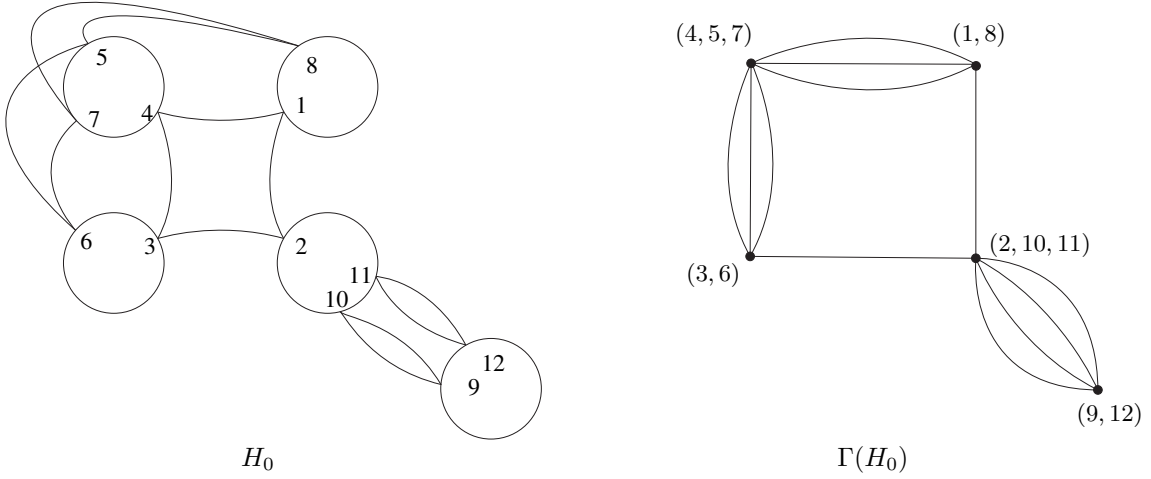
\begin{figure}[htbp]
\centering
\includegraphics[width=0.88\linewidth]{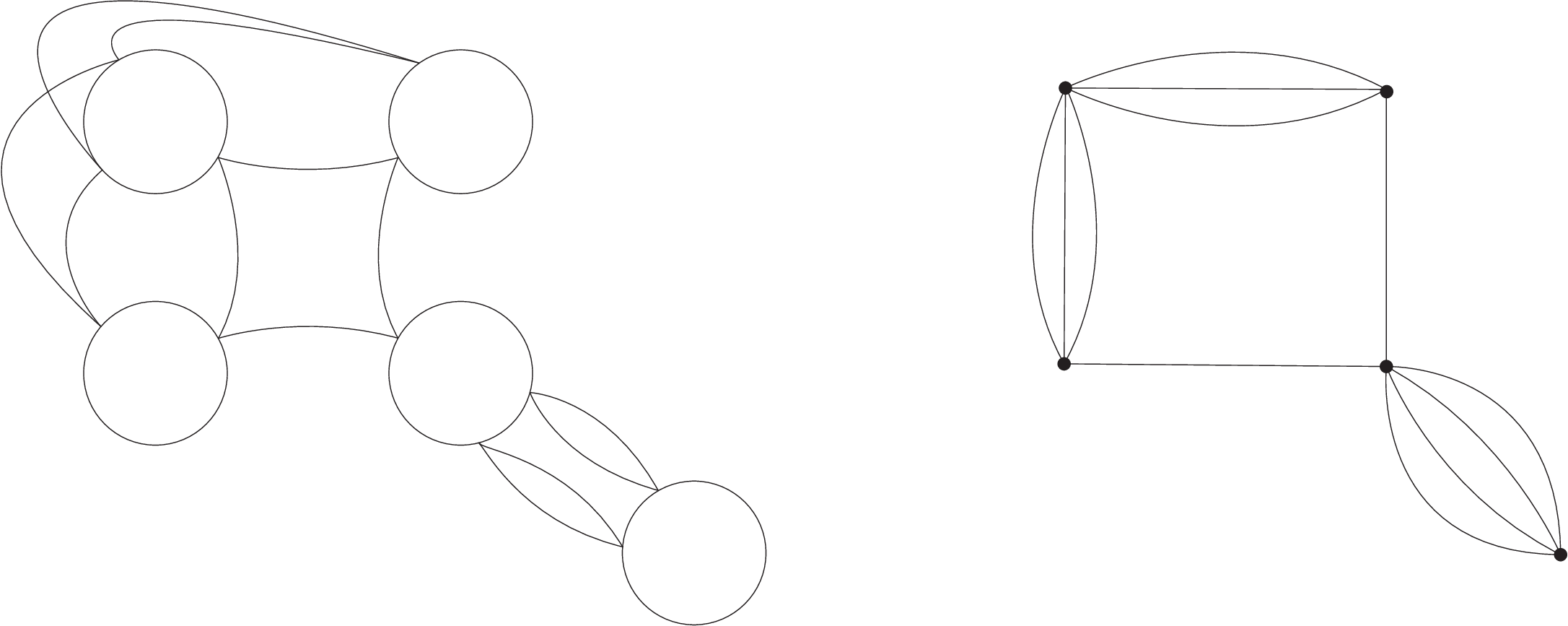}
\put(-304,123){\footnotesize1}
\put(-300,138){\footnotesize8}
\put(-362,120){\footnotesize4}
\put(-379,141){\footnotesize5}
\put(-382,118){\footnotesize7}
\put(-304,70){\footnotesize2}
\put(-284,59){\footnotesize11}
\put(-293,50){\footnotesize10}
\put(-361,70){\footnotesize3}
\put(-385,72){\footnotesize6}
\put(-234,26){\footnotesize12}
\put(-239,17){\footnotesize9}
\put(-161,150){\footnotesize$(4,5,7)$}
\put(-57,150){\footnotesize$(1,8)$}
\put(-157,60){\footnotesize$(3,6)$}
\put(-43,71){\footnotesize$(2,10,11)$}
\put(-10,7){\footnotesize$(9,12)$}
\put(-324.5,-10){$H_0$}
\put(-100,-10){$\Gamma(H_0)$}
\caption{A hypermap $H_0$ and its vertex-adjacency graph $\Gamma(H_0)$.}
\label{fig:hypermap-H0}
\end{figure}

We will use the hypermap $H_0$ as a running example throughout the subsequent sections.

\begin{proposition}\label{prop:adj-bip}
The hypermap $H$ is bipartite if and only if there exists a 2-coloring
\[
\chi:V(H)\to\{1,2\}
\]
such that for every dart $i\in D$ one has
\[
\chi(v(i))\neq \chi(v(\alpha(i))).
\]
Equivalently, if $e=(i_1i_2\cdots i_k)$ is a hyperedge, then the vertex colors alternate along the cyclic order of $e$:
\[
\chi(v(i_j))\neq \chi(v(i_{j+1}))
\qquad (1\le j\le k,\ i_{k+1}:=i_1).
\]
\end{proposition}

\begin{proof}
By definition, each edge of $\Gamma(H)$ is induced by some dart $i$ and joins $v(i)$ to $v(\alpha(i))$. Therefore $\Gamma(H)$ is 2-colorable if and only if there exists such a map $\chi$ making the endpoints of every such edge differently colored. If $e=(i_1\cdots i_k)$, then $\alpha(i_j)=i_{j+1}$, so the condition is exactly the alternation of colors along the cyclic order of the hyperedge. The converse is immediate. In particular, if there exists some $i$ with $v(i)=v(\alpha(i))$, then $\Gamma(H)$ contains a loop and $H$ cannot be bipartite.
\end{proof}

\subsection{A two-permutation formula for hyperedge partial duality}

The notion of partial duality follows Chmutov and Vignes-Tourneret \cite[Definition~2.1]{ChmutovVignes2022}. More generally, if $S$ is a set of cells of the same type (vertices, hyperedges, or faces) in a hypermap, then one can define the partial dual $H^S$ with respect to $S$; when $S$ is the set of all cells of that type, one recovers the total dual. In this paper we need only the case of a subset $E'\subseteq E(H)$ of hyperedges. Geometrically, one may regard $H^{E'}$ as obtained in four steps: first take the surface with boundary formed by all vertex-cells together with the selected hyperedge-cells; then cap each boundary component of this surface by a new vertex-cell; next take a copy of each hyperedge of the original hypermap and glue these copies to the new vertices along the boundary intervals they share with the intermediate surface; finally cap the remaining boundary components by face-cells. By \cite[Lemma~2.2 and Lemma~2.3(b)]{ChmutovVignes2022}, this construction is independent of the auxiliary choice of cell type and preserves the natural correspondence between all hyperedges together with their lengths.

In the following, we give an explicit formula in the two-permutation model. We first recall Chmutov and Vignes–Tourneret's characterization of orientable hypermaps in the three-permutation model. Let $E'\subseteq E(H)$ be a hyperedge subset, and write $\overline{E'}:=E(H)\setminus E'$. 

\begin{theorem}[Chmutov and Vignes-Tourneret {\cite{ChmutovVignes2022}}]
\label{thm:CVT-partial-duality}
Let $H=(\delta_V,\delta_E,\delta_F)$ be an orientable hypermap in the three-permutation model, where $\delta_F\delta_E\delta_V=1$. Then the partial dual $H^{E'}$
is given as follows:
\[
        \bigl(
            \delta_{E'}\delta_V,\,
            \delta_{\overline{E'}}\delta_{E'}^{-1},\,
            \delta_F\delta_{E'}
        \bigr).
\]
\end{theorem}

Let $\alpha_{E'}$ denote the product of the $\alpha$-cycles corresponding to hyperedges in $E'$, and let $\alpha_{\overline{E'}}$ denote the product of the complementary hyperedge cycles. These two permutations have disjoint supports, hence commute, and satisfy
\[
\alpha=\alpha_{E'}\alpha_{\overline{E'}}.
\]

\begin{proposition}\label{prop:partialdual}
Let $H=(\sigma,\alpha)$ be an orientable hypermap, and let $E'\subseteq E(H)$. Then the partial dual $H^{E'}$ with respect to $E'$ can be written in the $(\sigma,\alpha)$ model as
\[
H^{E'}=(\sigma^{E'},\alpha^{E'}),
\]
where
\begin{equation}\label{eq:partialdual}
\sigma^{E'}=\alpha_{E'}^{-1}\sigma,
\qquad
\alpha^{E'}=\alpha_{E'}^{-1}\alpha_{\overline{E'}}.
\end{equation}
Equivalently, $\alpha^{E'}$ is obtained from $\alpha$ by reversing all cycles corresponding to selected hyperedges and leaving the remaining cycles unchanged.
\end{proposition}

\begin{proof}
Write $H$ in the three-permutation notation $(\delta_V,\delta_E,\delta_F)$ of Chmutov and Vignes-Tourneret, where $\delta_V$ records vertices and $\delta_E$ records hyperedges. After translating to the present two-permutation notation, we may take $\delta_V=\sigma$, while $\delta_E$ is the reversal of the cycles of $\alpha$. Let $\delta_{E'}$ be the product of the cycles of $\delta_E$ corresponding to hyperedges in $E'$, and let $\delta_{\overline{E'}}$ be the product of the complementary cycles. Then
\[
\delta_{E'}=\alpha_{E'}^{-1},
\qquad
\delta_{\overline{E'}}=\alpha_{\overline{E'}}^{-1}.
\]
By Theorem~\ref{thm:CVT-partial-duality}, the partial dual with respect to $E'$ transforms the vertex permutation into
\[
\sigma^{E'}=\delta_{E'}\delta_V=\alpha_{E'}^{-1}\sigma.
\]
At the same time, the new hyperedge permutation is obtained exactly by reversing the selected cycles and leaving the unselected ones unchanged, that is,
\[
\alpha^{E'}=\alpha_{E'}^{-1}\alpha_{\overline{E'}}.
\]
This is precisely \eqref{eq:partialdual}.
\end{proof}

\paragraph{Supplementary remark.}
If one wishes to match Proposition~\ref{prop:partialdual} term by term with the three-permutation model, then the identity
\[
\delta_F\delta_E\delta_V=1
\]
implies
\[
\phi^{-1}\alpha^{-1}\sigma=1
\qquad\Longleftrightarrow\qquad
\phi=\alpha^{-1}\sigma,
\]
where $\phi$ is the face permutation, and hence $\delta_F=\phi^{-1}=\sigma^{-1}\alpha$. Therefore, by equation~\eqref{eq:partialdual}, the face permutation after partial duality is
\[
\phi^{E'}:=(\alpha^{E'})^{-1}\sigma^{E'}
=(\alpha_{E'}^{-1}\alpha_{\overline{E'}})^{-1}(\alpha_{E'}^{-1}\sigma)
=\alpha_{\overline{E'}}^{-1}\sigma,
\]
so that
\[
\begin{aligned}
(\phi^{E'})^{-1}
&=(\sigma^{E'})^{-1}\alpha^{E'} \\
&=(\alpha_{E'}^{-1}\sigma)^{-1}(\alpha_{E'}^{-1}\alpha_{\overline{E'}}) \\
&=\sigma^{-1}\alpha_{\overline{E'}}
=\sigma^{-1}\alpha\alpha_{E'}^{-1}
=\delta_F\delta_{E'}.
\end{aligned}
\]
which agrees exactly with the third component in the Chmutov--Vignes-Tourneret formula. We separate this derivation from the main proof only for clarity.

\begin{corollary}\label{cor:length-preserved}
The hyperedges of $H$ and $H^{E'}$ have the same multiset of lengths under the natural correspondence. More precisely, the cycle lengths of $\alpha$ and $\alpha^{E'}$ coincide exactly.
\end{corollary}
\begin{proof}
Equation~\eqref{eq:partialdual} shows that $\alpha^{E'}$ is obtained merely by reversing the cycles of $\alpha$ corresponding to hyperedges in $E'$, without changing the support or the length of any cycle.
\end{proof}

\begin{corollary}\label{cor:odd-obstruction}
If there exists $E'\subseteq E(H)$ such that $H^{E'}$ is bipartite, then every hyperedge of $H$ must have even length.
\end{corollary}
\begin{proof}
If $H^{E'}$ is bipartite, then by Proposition~\ref{prop:adj-bip} the vertex colors along every hyperedge of $H^{E'}$ must alternate, so every hyperedge of $H^{E'}$ has even length. By Corollary~\ref{cor:length-preserved}, $H$ and $H^{E'}$ have the same hyperedge lengths, and the claim follows.
\end{proof}

\begin{remark}\label{rem:total-dual}
When $E'=E(H)$, equation~\eqref{eq:partialdual} reduces to
\[
\sigma^{E(H)}=\alpha^{-1}\sigma=\phi,
\qquad
\alpha^{E(H)}=\alpha^{-1},
\]
which is exactly the total dual with respect to all hyperedges: the new vertices are the old faces, while the hyperedges are preserved with reversed orientation \cite{ChmutovVignes2022}.
\end{remark}

\section{The medial map, smoothings, and state circles}

\subsection{Permutation definition of the medial map}

We use the medial-map construction of Cori and Hetyei \cite[Definition~3.1]{CoriHetyei2025}. Let
\[
D^{\pm}=\{1^-,1^+,\dots,n^-,n^+\}.
\]
If a cycle of $\alpha$ is $(i_1i_2\cdots i_k)$, then we include in $\sigma_M$ the cycle
\[
(i_1^-\,i_1^+\,i_2^-\,i_2^+\,\cdots\,i_k^-\,i_k^+).
\]
We further define
\[
\alpha_M:=\prod_{i\in D}(i^+\,\sigma(i)^-).
\]
Then $M(H)=(\sigma_M,\alpha_M)$ is an embedded map whose vertices correspond bijectively to the hyperedges of the original hypermap, while its edges correspond bijectively to the darts of the original hypermap.

\begin{lemma}\label{lem:checkerboard}
Let $\phi_M:=\alpha_M^{-1}\sigma_M=\alpha_M\sigma_M$ denote the face permutation of the medial map $M(H)$ associated with hypermap $H=(\sigma,\alpha)$. Then:
\begin{enumerate}
  \item the restriction of $\phi_M$ to the negative points $\{i^-:i\in D\}$ is isomorphic to $\sigma$;
  \item the restriction of $\phi_M$ to the positive points $\{i^+:i\in D\}$ is isomorphic to $\sigma^{-1}\alpha$, and hence has the same cycle decomposition as the original face permutation $\phi=\alpha^{-1}\sigma$;
  \item the medial map \(M(H)\) admits a canonical checkerboard coloring: the
\(\phi_M\)-cycles contained in \(D^-\) are colored black, and the
\(\phi_M\)-cycles contained in \(D^+\) are colored white. 
\end{enumerate}
\end{lemma}

\begin{proof}
First take a negative point \(i^-\). By the definition of \(\sigma_M\), we have
\[
    \sigma_M(i^-)=i^+.
\]
Since \(\alpha_M\) contains the transposition \((i^+\,\sigma(i)^-)\), it follows that
\[
    \phi_M(i^-)
    =\alpha_M\sigma_M(i^-)
    =\alpha_M(i^+)
    =\sigma(i)^-.
\]
Thus the restriction of \(\phi_M\) to \(D^-\) is identified with \(\sigma\).

Now take a positive point \(i^+\). Write \(\alpha(i)=j\). By the construction of
\(\sigma_M\), the point following \(i^+\) around the corresponding medial vertex is
\(j^-\). Hence
\[
    \sigma_M(i^+)=j^-=\alpha(i)^-.
\]
Moreover, since \(\alpha_M\) contains the transposition
\[
    \bigl(\sigma^{-1}(j)^+\,j^-\bigr),
\]
we obtain
\[
    \phi_M(i^+)
    =\alpha_M\sigma_M(i^+)
    =\alpha_M(j^-)
    =\sigma^{-1}(j)^+
    =\sigma^{-1}\alpha(i)^+.
\]
Therefore the restriction of \(\phi_M\) to \(D^+\) is identified with
\(\sigma^{-1}\alpha\). Since
\[
    \sigma^{-1}\alpha=(\alpha^{-1}\sigma)^{-1}=\varphi^{-1},
\]
this restriction has the same cycle decomposition as the original face permutation
\(\phi=\alpha^{-1}\sigma\), with the cyclic order of each cycle reversed.

By parts (1) and (2), the face permutation
\[
        \phi_M=\alpha_M^{-1}\sigma_M=\alpha_M\sigma_M
\]
satisfies
\[
        \phi_M(D^-)\subseteq D^-,
        \qquad
        \phi_M(D^+)\subseteq D^+ .
\]
Hence every face-cycle of \(M(H)\) is contained entirely in either
\(D^-\) or \(D^+\). We color the face-cycles contained in \(D^-\) black
and the face-cycles contained in \(D^+\) white.

We now describe this coloring locally at a vertex of the medial map. Let
\[
        e=(i_1\,i_2\,\cdots\,i_k)
\]
be a hyperedge of \(H\). By the definition of the medial map, the
corresponding vertex \(v_e\) of \(M(H)\) has cyclic order
\[
        i_1^-\, i_1^+\, i_2^-\, i_2^+\,\cdots\, i_k^-\, i_k^+ .
\]
Thus the corners at \(v_e\) are of two alternating forms:
\[
        (i_j^-,i_j^+)
        \qquad\text{and}\qquad
        (i_j^+,i_{j+1}^-),
        \quad i_{k+1}:=i_1 .
\]

We claim that the corners of the first form are black corners and the
corners of the second form are white corners. Indeed, for a negative
point \(i_j^-\), the face traversal gives
\[
        i_j^-
        \xrightarrow{\ \sigma_M\ }
        i_j^+
        \xrightarrow{\ \alpha_M\ }
        \sigma(i_j)^- .
\]
Thus, when the corresponding face passes through \(v_e\), it passes
through the local corner \((i_j^-,i_j^+)\). Since the corresponding
\(\phi_M\)-cycle is contained in \(D^-\), this is a black face-corner.

On the other hand, for a positive point \(i_j^+\), we have
\[
        i_j^+
        \xrightarrow{\ \sigma_M\ }
        i_{j+1}^-
        \xrightarrow{\ \alpha_M\ }
        \sigma^{-1}(i_{j+1})^+
        =
        \sigma^{-1}\alpha(i_j)^+ .
\]
Thus, when the corresponding face passes through \(v_e\), it passes
through the local corner
\[
        (i_j^+,i_{j+1}^-)
        =
        (i_j^+,\alpha(i_j)^-).
\]
Since the corresponding \(\phi_M\)-cycle is contained in \(D^+\), this is
a white face-corner. 

Thus black and white face corners alternate at every vertex of \(M(H)\). Accordingly, such black-and-white coloring yields a checkerboard coloring of \(M(H)\).
\end{proof}

Figure~\ref{fig:checkerboard-MH0} illustrates this mechanism for the example
$M(H_0)$. In that example, the black faces are exactly the face-cycles contained in \(D^-\),
namely
\[
    (1^-\,8^-),\quad
    (2^-\,10^-\,11^-),\quad
    (3^-\,6^-),\quad
    (4^-\,5^-\,7^-),\quad
    (9^-\,12^-).
\]
The white faces are exactly the face-cycles contained in \(D^+\), namely
\[
    (1^+\,11^+\,9^+\,2^+\,6^+\,5^+\,3^+\,7^+),\quad
    (4^+\,8^+),\quad
    (10^+\,12^+).
\] 

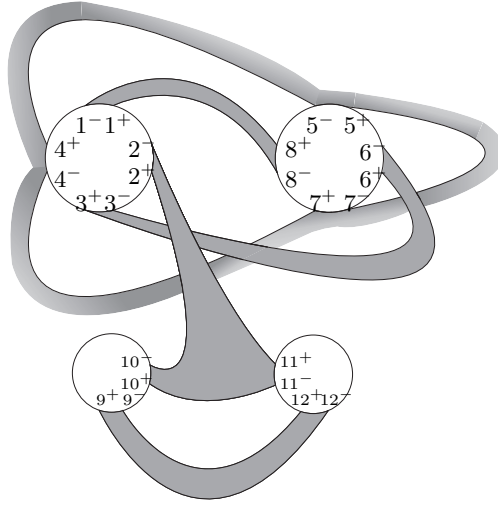
\begin{figure}[htbp]
\centering
\includegraphics[width=0.4\linewidth]{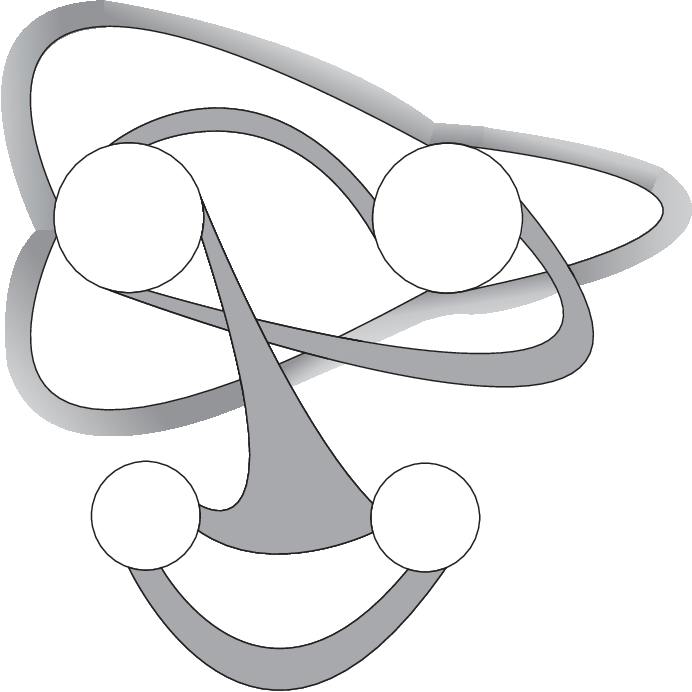}
\put(-162,138){\footnotesize$1^-$}
\put(-151,138){\footnotesize$1^+$}
\put(-142,129){\footnotesize$2^-$}
\put(-142,119){\footnotesize$2^+$}
\put(-151,109){\footnotesize$3^-$}
\put(-162,109){\footnotesize$3^+$}
\put(-170,118){\footnotesize$4^-$}
\put(-170,129){\footnotesize$4^+$}
\put(-75,138){\footnotesize$5^-$}
\put(-61,138){\footnotesize$5^+$}
\put(-55,128){\footnotesize$6^-$}
\put(-55,118){\footnotesize$6^+$}
\put(-61,109){\footnotesize$7^-$}
\put(-74,109){\footnotesize$7^+$}
\put(-83,118){\footnotesize$8^-$}
\put(-83,129){\footnotesize$8^+$}
\put(-144,36){\tiny$9^-$}
\put(-154,36){\tiny$9^+$}
\put(-145,50){\tiny$10^-$}
\put(-145,42){\tiny$10^+$}
\put(-85,42){\tiny$11^-$}
\put(-85,50){\tiny$11^+$}
\put(-70,36){\tiny$12^-$}
\put(-81,36){\tiny$12^+$}
\caption{Canonical checkerboard coloring of the medial map \(M(H_0)\).}
\label{fig:checkerboard-MH0}
\end{figure}

\subsection{Black/white smoothings and states}

We keep the canonical checkerboard coloring of $M(H)$ fixed as in Lemma~3.1.
Thus, at the medial vertex $v_e$ corresponding to a hyperedge
\[
e=(i_1 i_2 \cdots i_k),
\]
the black and white face-corners are the alternating corners described there:
\[
(i_j^-,i_j^+)
\quad\text{and}\quad
(i_j^+,i_{j+1}^-),
\qquad i_{k+1}:=i_1,
\]
respectively.

\begin{definition}
At the vertex $v_e$ of $M(H)$, the \emph{black smoothing} is the local
replacement which connects, for every $j$, the two half-edges incident with
the black corner $(i_j^-,i_j^+)$. Similarly, the \emph{white smoothing} is the
local replacement which connects, for every $j$, the two half-edges incident
with the white corner $(i_j^+,i_{j+1}^-)$.

Equivalently, each smoothing removes the vertex $v_e$ and replaces it by
$k$ degree-$2$ vertices, one for each corner of the chosen color. These local
pairings are non-crossing in the cyclic order around $v_e$.
\end{definition}

Let $E'\subseteq E(H)$ be a set of hyperedges. The state $S_{E'}$ is obtained
from $M(H)$ by applying the white smoothing at $v_e$ when $e\in E'$, and the
black smoothing at $v_e$ when $e\notin E'$. After all medial vertices are
smoothed, the resulting graph is $2$-regular. Its connected components are
called \emph{state circles}. The length of a state circle is the number of
medial edges contained in it.

An example is presented in Figure~\ref {fig:smoothing-AB}. Vertices of degree 2 are only added to illustrate the black and white smoothing operations, and no additional degree-2 vertices will be introduced for subsequent figures when black or white smoothing is applied.

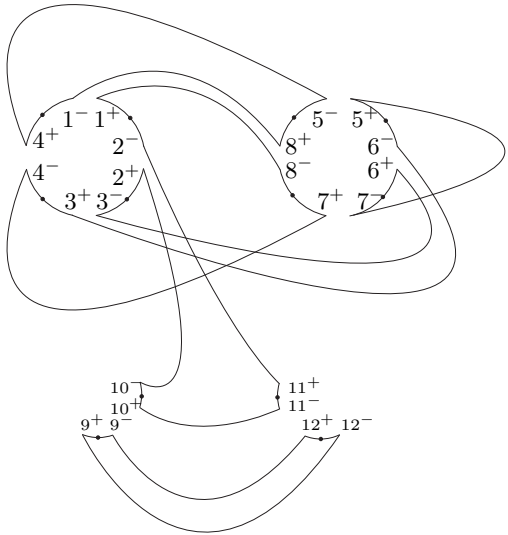
\begin{figure}[htbp]
\centering
\includegraphics[width=0.4\linewidth]{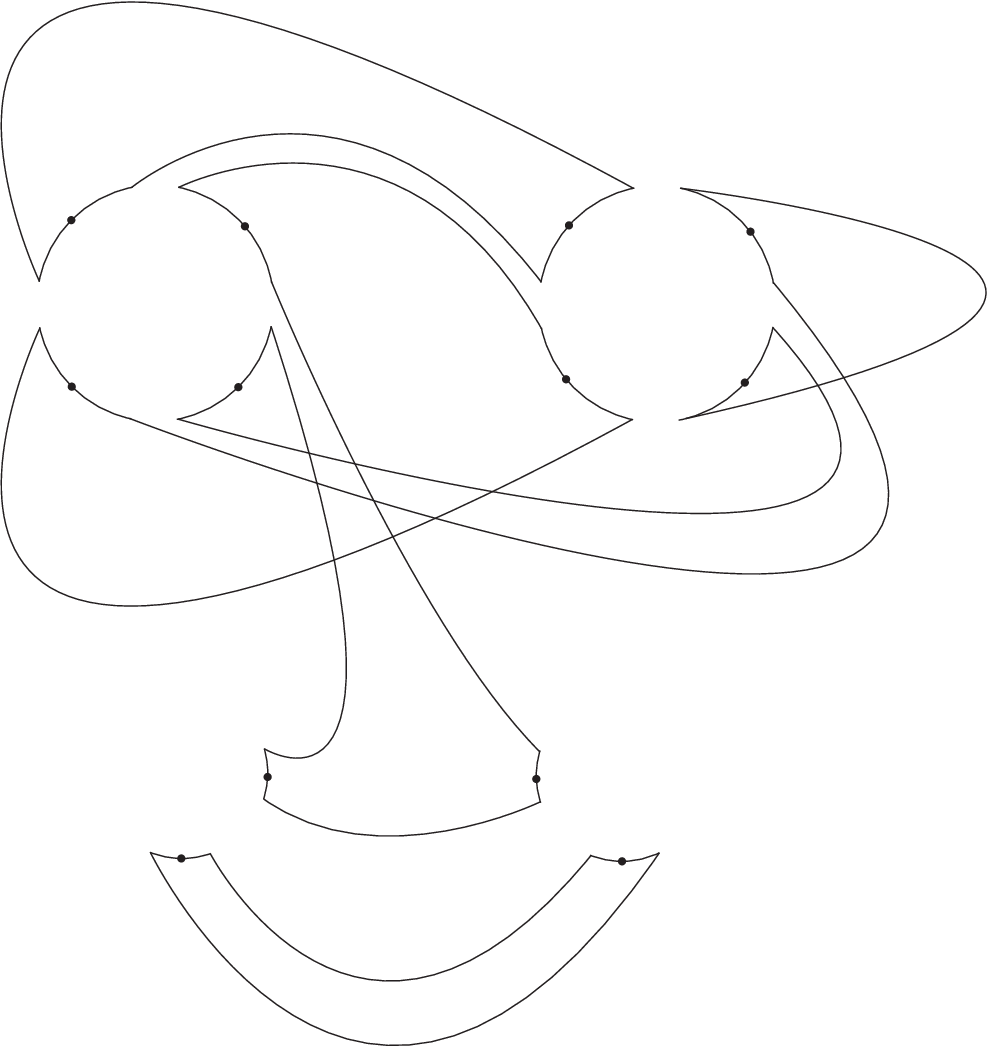}
\put(-167,153){\footnotesize$1^-$}
\put(-155,153){\footnotesize$1^+$}
\put(-148,143){\footnotesize$2^-$}
\put(-148,131){\footnotesize$2^+$}
\put(-154,122){\footnotesize$3^-$}
\put(-166,122){\footnotesize$3^+$}
\put(-178,133){\footnotesize$4^-$}
\put(-178,145){\footnotesize$4^+$}
\put(-73,153){\footnotesize$5^-$}
\put(-58,153){\footnotesize$5^+$}
\put(-52,143){\footnotesize$6^-$}
\put(-52,134){\footnotesize$6^+$}
\put(-56,122){\footnotesize$7^-$}
\put(-71,122){\footnotesize$7^+$}
\put(-83,134){\footnotesize$8^-$}
\put(-83,143){\footnotesize$8^+$}
\put(-149,39){\tiny$9^-$}
\put(-160,39){\tiny$9^+$}
\put(-149,53){\tiny$10^-$}
\put(-149,45){\tiny$10^+$}
\put(-82,46){\tiny$11^-$}
\put(-82,53){\tiny$11^+$}
\put(-62,39){\tiny$12^-$}
\put(-77,39){\tiny$12^+$}
\caption{The $S_{\{A,B\}}$ state of the medial map \(M(H_0)\).}
\label{fig:smoothing-AB}
\end{figure}

\section{The bijection between state circles and vertices of a partial dual}

The key correspondence of the paper is established in this section: the state circles of $S_{E'}$ are precisely the vertex-cycles of $H^{E'}$.

Recall that $\sigma$ and $\alpha$ stand for the two permutations of the original hypermap $H=(\sigma,\alpha)$ defined on its dart set $D$, while $\sigma_M$ and $\alpha_M$ denote the permutations associated with its medial map. For any dart $d\in D$, $e(d)$ represents the $\alpha$-cycle of $H$ containing $d$, namely the hyperedge incident to $d$. Accordingly, for every $i\in D$, $\sigma(i)$ is also a dart of $H$, and $e(\sigma(i))$ refers to the hyperedge (or equivalently, the $\alpha$-cycle) containing $\sigma(i)$.

For each dart $i\in D$, define \(m_i = (i^+, \sigma(i)^-)\) as the edge of the medial map $M(H)$ corresponding to $i$. Here we label the edge $m_i$ by the dart $i$.

\begin{proposition}\label{prop:state-successor}
Let \(E'\subseteq E(H)\). Define a map \(\pi_{E'}: D \to D\) as follows: $\pi_{E'}(i)$ denotes the index of the medial edge succeeding $m_i$ when traversing the medial edge $m_i$ from $i^+$ to $\sigma(i)^-$ along the state circle. Then
\[
\pi_{E'}(i) =
\begin{cases}
\sigma(i), & e(\sigma(i)) \notin E',\\
\alpha^{-1}(\sigma(i)), & e(\sigma(i))\in E',
\end{cases}
\tag{4.1}
\]
and hence
\[
\pi_{E'} = \alpha^{-1}_{E'}\sigma.
\]
\end{proposition}
\begin{proof}
Fix a dart \(i\in D\) and put \(d:=\sigma(i)\). Then the endpoint reached is \(d^-\). This signed point lies at the medial vertex corresponding to
the hyperedge \(e(d)\) of the original hypermap.

We split the argument into two cases based on whether the hyperedge \(e(d)\) lies in \(E'\).

\medskip
\noindent
\textbf{Case 1: \(e(d)\notin E'\).}
\medskip

In state \(S_{E'}\), the vertex \(v_{e(d)}\) is equipped with a black smoothing. By the local characterization of black corners in Lemma 3.1, the unique black corner containing \(d^-\) is \((d^-,d^+)\). As a result, the black smoothing matches the half-edge at \(d^-\) with the half-edge at \(d^+\), and the path departs from the vertex via \(d^+\) after crossing this smoothing block.

The medial edge having positive terminal \(d^+\) is precisely \(m_d=(d^+,\sigma(d)^-)\). Hence the medial edge immediately following \(m_i\) is \(m_d\), leading to \(\pi_{E'}(i)=d=\sigma(i)\).

\medskip
\noindent
\textbf{Case 2: \(e(d)\in E'\).}
\medskip

Here the vertex \(v_{e(d)}\) is assigned a white smoothing under state \(S_{E'}\). Set \(r:=\alpha^{-1}(d)\); then \(\alpha(r)=d\), meaning the dart \(d\) is the immediate successor of \(r\) in the cyclic ordering of the hyperedge \(e(d)\). Lemma 3.1 implies the white corner associated with \(r\) equals \((r^+,\alpha(r)^-)=(r^+,d^-)\). Accordingly, the white smoothing pairs the half-edge incident to \(d^-\) with the half-edge incident to \(r^+\).

After traversing the smoothing block, the path exits the vertex through \(r^+\). The medial edge with positive terminal \(r^+\) is \(m_r=(r^+,\sigma(r)^-)\), so \(m_r\) is the medial edge succeeding \(m_i\). Substituting \(r=\alpha^{-1}(d)=\alpha^{-1}(\sigma(i))\), we conclude \(\pi_{E'}(i)=\alpha^{-1}(\sigma(i))\).

\end{proof}

\begin{theorem}\label{thm:state-vertex}
The state circles of $S_{E'}$ are in bijection with the vertices of the partial dual $H^{E'}$. More precisely, the darts belonging to a given state circle form a cycle of the permutation $\sigma^{E'}$, and the length of the state circle equals the length of that cycle.
\end{theorem}

\begin{proof}
By Proposition~\ref{prop:partialdual}, the vertex permutation of $H^{E'}$ is \(\sigma^{E'}=\alpha_{E'}^{-1}\sigma\). By Proposition~\ref{prop:state-successor}, the successor map along a state circle is also $\pi_{E'}=\alpha_{E'}^{-1}\sigma$. Therefore the state circles are exactly the cycles of $\sigma^{E'}$. The number of medial edges on a state circle is the number of darts in the corresponding cycle, which is precisely the degree of the corresponding vertex of $H^{E'}$.
\end{proof}

\begin{corollary}\label{cor:even-state}
The partial dual $H^{E'}$ is Eulerian if and only if every state circle of $S_{E'}$ has even length.
\end{corollary}

\begin{proof}
By Theorem~\ref{thm:state-vertex}, the degrees of the vertices of $H^{E'}$ are exactly the lengths of the corresponding state circles, so the two statements are equivalent.
\end{proof}

\section{Eulerian partial duals: A crossing-total characterization}

In this section we orient the edges of $M(H)$ arbitrarily. For a corner at a vertex, if the two incident edges are both directed into the vertex or both directed out of the vertex, then we call that corner a \emph{source/sink corner}; otherwise it is a \emph{crossing corner}.

\begin{definition}\label{def:crossing-total}
Let $v$ be a vertex of a checkerboard-colored oriented map.
\begin{enumerate}
  \item If all edges incident with $v$ are directed into $v$, or all are directed out of $v$, then $v$ is called \emph{type-$t$}.
  \item If $v$ is not type-$t$ and every black corner at $v$ is a source/sink corner, then $v$ is called \emph{type-$c$}.
  \item If $v$ is not type-$t$ and every white corner at $v$ is a source/sink corner, then $v$ is called \emph{type-$d$}.
\end{enumerate}
An orientation of the edges of $M(H)$ is called a \emph{crossing-total direction} if every vertex of $M(H)$ is of one of the above three types.
\end{definition}

\begin{remark}\label{rem:types-disjoint}
The three types in Definition~5.1 are pairwise disjoint. Indeed, if a vertex
simultaneously satisfies ``all black corners are source/sink corners'' and
``all white corners are source/sink corners'', then every pair of consecutive
incident edges around the vertex is either both directed into the vertex or both
directed out of the vertex. Hence all incident edges have the same in/out status
at the vertex, and the vertex must be $t$-type.
\end{remark}

An illustrative example is shown in Figure~\ref{fig:local-types}. In the left figure, all incident edges are directed into the vertex, making the vertex type-t; in the middle figure, all black corners are source/sink corners, making the vertex type-c; and in the right figure, all white corners are source/sink corners, making the vertex type-d.

\begin{figure}[htbp]
\centering
\includegraphics[width=0.6\linewidth]{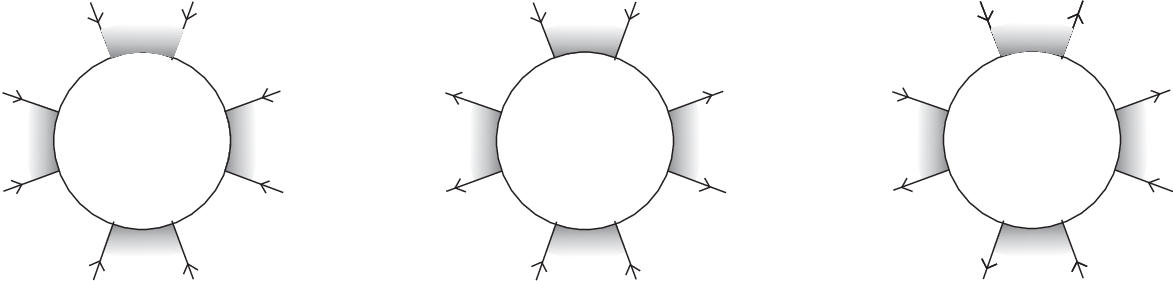}
\put(-275,-10){Type-$t$ vertex}
\put(-168,-10){Type-$c$ vertex}
\put(-66,-10){Type-$d$ vertex}
\caption{Type-$c$, $d$, and $t$ vertices.}
\label{fig:local-types}
\end{figure}

For a crossing-total direction $\Omega$, let $D(\Omega)$ denote the set of all type-$d$ hyperedges and let $T(\Omega)$ denote the set of all type-$t$ hyperedges.

\begin{theorem}\label{thm:eulerian-char}
Let $H=(\sigma,\alpha)$ be an orientable hypermap, and let $E'\subseteq E(H)$. The following are equivalent:
\begin{enumerate}
  \item $H^{E'}$ is Eulerian.
  \item There exists a crossing-total direction $\Omega$ of $M(H)$ such that
  \[
  E'=D(\Omega)\cup T',\qquad T'\subseteq T(\Omega).
  \]
\end{enumerate}
\end{theorem}

\begin{proof}
By Corollary~\ref{cor:even-state}, statement (1) is equivalent to saying that all state circles of $S_{E'}$ have even length. It therefore suffices to prove
\[
S_{E'}\text{ is an even state}
\Longleftrightarrow
\exists\,\text{ crossing-total direction $\Omega$ such that }E'=D(\Omega)\cup T'.
\]

\noindent(1)$\Rightarrow$(2). Assume that $S_{E'}$ is an even state. For each state circle $C$, choose an arbitrary traversal direction and list its edges as
\[
f_1,f_2,\dots,f_{2m}.
\]
Orient them alternately by parity: direct $f_{2r-1}$ along the traversal direction and $f_{2r}$ opposite to it. Then every degree-$2$ vertex on the state circle becomes either a source or a sink.

Because smoothing changes only the local connections at vertices and does not alter the edge set itself, this orientation descends directly to an orientation $\Omega$ of the edges of $M(H)$. Let $e$ be any hyperedge and let $v_e$ be the corresponding medial vertex.
\begin{itemize}
  \item If $e\in E'$, then $v_e$ is white-smoothed. Each white smoothing block becomes a degree-$2$ vertex in the state, and each such vertex is a source or a sink. Hence every white corner of $v_e$ is a source/sink corner. If in addition all incident edges are co-oriented, then $v_e$ is type-$t$; otherwise it is type-$d$.
  \item If $e\notin E'$, then $v_e$ is black-smoothed. The same argument shows that every black corner is a source/sink corner, so $v_e$ is either type-$c$ or type-$t$.
\end{itemize}
Thus $\Omega$ is a crossing-total direction, and
\[
e\in E'\Rightarrow e\in D(\Omega)\cup T(\Omega),
\qquad
e\notin E'\Rightarrow e\notin D(\Omega)\setminus T(\Omega).
\]
Therefore every type-$d$ hyperedge belongs to $E'$, while the type-$t$ hyperedges may be chosen freely. Setting
\[
T':=E'\cap T(\Omega)
\]
gives $E'=D(\Omega)\cup T'$.

\noindent(2)$\Rightarrow$(1). Assume that there exists a crossing-total direction $\Omega$ such that
\[
E'=D(\Omega)\cup T',\qquad T'\subseteq T(\Omega).
\]
Construct the state $S_{E'}$ by taking white smoothing for $e\in E'$ and black smoothing for $e\notin E'$. By assumption:
\begin{itemize}
  \item if $e\in D(\Omega)$, then all white corners are source/sink corners;
  \item if $e\notin E'$ and $e$ is type-$c$, then all black corners are source/sink corners;
  \item if $e$ is type-$t$, then regardless of whether one chooses black or white smoothing, each smoothing block clearly yields a source or a sink.
\end{itemize}
Hence every degree-$2$ vertex produced by smoothing is a source or a sink.

Now take any state circle $C$ and traverse its edges in some direction, writing them as
\[
g_1,g_2,\dots,g_{\ell}.
\]
Define
\[
s_r=
\begin{cases}
+1,& \text{if the traversal direction agrees with the orientation of }g_r,\\
-1,& \text{if the traversal direction is opposite to the orientation of }g_r.
\end{cases}
\]
If $g_r$ and $g_{r+1}$ meet at a degree-$2$ vertex $u$, then $u$ is a source or
a sink. Hence the two edges are both directed into $u$ or both directed out of $u$.
Since one edge enters and the other leaves during the traversal through $u$, we must have
\[s_{r+1}=-s_r .\]

Thus recursively
\[
s_{\ell+1}=(-1)^{\ell}s_1.
\]
But after one full turn of the closed walk we must have $s_{\ell+1}=s_1$, hence $(-1)^{\ell}=1$ and therefore $\ell$ is even. So every state circle has even length, that is, $S_{E'}$ is an even state. Corollary~\ref{cor:even-state} now implies that $H^{E'}$ is Eulerian.
\end{proof}

\section{Bipartite partial duals: an all-crossing characterization}

\subsection{\texorpdfstring{$x/y$}{x/y} labels and all-crossing directions}

Fix the orientation of the underlying surface and the checkerboard coloring from Lemma~\ref{lem:checkerboard}. Given any orientation of the edges of $M(H)$, define for each oriented edge $f$ the label
\[
\lambda(f)=
\begin{cases}
x,& \text{if the black face lies on the right of }f,\\
y,& \text{if the black face lies on the left of }f.
\end{cases}
\]

The left panel of Figure~\ref{fig:xy-corner-labels} shows the labels when both edges are directed into the black corner $u$, whereas the right panel shows them when one edge points into $u$ and the other out of it.

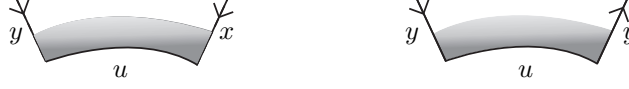
\begin{figure}[htbp]
\centering
\includegraphics[width=0.5\linewidth]{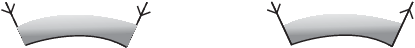}
\put(-234,10){$y$}
\put(-155,10){$x$}
\put(-85,10){$y$}
\put(-3,10){$y$}
\put(-195,-4){$u$}
\put(-42.5,-4){$u$}
\caption{$x$ and $y$ labels at a black corner $u$.}
\label{fig:xy-corner-labels}
\end{figure}

\begin{lemma}\label{lem:xy-angle}
Let $u$ be a vertex, and let $a,b$ be two consecutive edges forming a corner at $u$.
Then
\[\text{$a,b$ are both directed into $u$ or both directed out of $u$}
\quad\Longleftrightarrow\quad
\lambda(a)\ne \lambda(b).\]
Equivalently,
\[\text{one of $a,b$ is directed into $u$ and the other is directed out of $u$}
\quad\Longleftrightarrow\quad
\lambda(a)=\lambda(b).\]
\end{lemma}
\begin{proof}
This is a local observation in an oriented neighborhood of the corner. If $a$ and
$b$ are both directed into $u$, or both directed out of $u$, then when one looks
along the two edge directions, the common incident face is seen with opposite
left/right positions. Hence the condition that the black face lies on the left or
on the right is reversed, and the two labels are different.

If, on the other hand, one of $a,b$ is directed into $u$ and the other is directed out of $u$, then the local left/right relation with respect to the common incident face is preserved. Hence the two labels are the same.
\end{proof}

\begin{definition}\label{def:all-crossing}
An orientation of the edges of $M(H)$ is called \emph{all-crossing} if, at every vertex, exactly one of the following holds:
\begin{enumerate}
  \item all black corners are source/sink corners and all white corners are crossing corners;
  \item all white corners are source/sink corners and all black corners are crossing corners.
\end{enumerate}
In case (1) the vertex is called a \emph{type-$c$ all-crossing vertex}; in case (2) it is called a \emph{type-$d$ all-crossing vertex}. For an all-crossing direction $\Phi$, let $C(\Phi)$ denote the set of all type-$c$ hyperedges.
\end{definition}

\begin{remark}\label{rem:allcrossing-ribbon}
When every hyperedge has length $2$, each medial vertex is $4$-valent, and Definition~\ref{def:all-crossing} reduces exactly to the standard notion of an all-crossing direction in the ribbon-graph literature \cite{HuggettMoffatt2013,DengJinMetsidik2020}.
\end{remark}

\subsection{Local adjacency of neighboring smoothing blocks around a hyperedge}

The following lemma is the indispensable local step in the proof of the bipartite theorem: it identifies adjacent smoothing blocks around a hyperedge with adjacency in the partial dual induced by a dart.

\begin{lemma}\label{lem:adjacent-blocks}
Let \(e=(i_1\,i_2\,\cdots\,i_k)\) be a hyperedge of \(H\), and let \(v_e\) be the corresponding vertex of the medial map \(M(H)\). Let \(H^{E'}=(\alpha_{E'}^{-1}\sigma,\alpha_{E'}^{-1}\alpha_{\overline{E'}})\) be the partial dual of \(H\) with respect to \(E'\).

\begin{enumerate}
\item If \(e\notin E'\), then \(S_{E'}\) uses the black smoothing at \(v_e\).
Let \(B_j\) be the smoothing block corresponding to the black corner
\(
(i_j^-,i_j^+),
\)
and let \(C_j\) be the state circle passing through \(B_j\). Then, for each
\(j\), the two vertices of \(H^{E'}\) corresponding to \(C_j\) and \(C_{j+1}\)
are adjacent by the edge \(e(i_j)\), indices are understood modulo \(k\).

\item If \(e\in E'\), then \(S_{E'}\) uses the white smoothing at \(v_e\).
Let \(W_j\) be the smoothing block corresponding to the white corner
\(
(i_j^+,i_{j+1}^-),
\)
and let \(C_j\) be the state circle passing through \(W_j\). Then, for each
\(j\), the two vertices of \(H^{E'}\) corresponding to \(C_j\) and \(C_{j+1}\)
are adjacent by the edge \(e(i_{j+1})\). 
\end{enumerate}

Consequently, any two consecutive smoothing blocks around the
cyclic order of a hyperedge correspond to adjacent vertices of the
partial dual \(H^{E'}\).
\end{lemma}
\begin{proof}
We appeal to Theorem~4.2, which identifies the state circles of \(S_{E'}\) with
the cycles of the vertex permutation \(\sigma^{E'}\) of \(H^{E'}\).
Specifically, if a state circle contains the medial edge \(m_d\), then the
corresponding vertex of \(H^{E'}\) is the \(\sigma^{E'}\)-cycle containing the
dart \(d\).

Consider first the case \(e\notin E'\). In \(S_{E'}\), the vertex \(v_e\) is
smoothed black. The black smoothing block \(B_j\) contains the point
\(i_j^+\). Since the medial edge with positive endpoint \(i_j^+\) is
\[
m_{i_j}=(i_j^+,\sigma(i_j)^-),
\]
the state circle \(C_j\) contains \(m_{i_j}\). Consequently, the vertex of
\(H^{E'}\) corresponding to \(C_j\) is the one containing the dart \(i_j\).
Similarly, \(B_{j+1}\) contains \(i_{j+1}^+\), so \(C_{j+1}\) corresponds to
the vertex containing \(i_{j+1}\). Because \(e\notin E'\), the hyperedge cycle
\(e\) is not reversed under the partial dual, hence
\[
\alpha^{E'}(i_j)=\alpha(i_j)=i_{j+1}.
\]
Thus, in \(H^{E'}\), the edge \(e(i_j)\) connects the vertex containing
\(i_j\) to the vertex containing \(i_{j+1}\). These are precisely the
vertices corresponding to \(C_j\) and \(C_{j+1}\), respectively. This proves
part~(1).

Now assume \(e\in E'\). Then \(S_{E'}\) applies the white smoothing at
\(v_e\). The white smoothing block \(W_j\) contains \(i_j^+\), so, as before,
\(W_j\) lies on a state circle \(C_j\) that contains the medial edge
\(m_{i_j}\), and therefore \(C_j\) corresponds to the vertex of \(H^{E'}\)
containing \(i_j\). Likewise, \(W_{j+1}\) contains \(i_{j+1}^+\), so
\(C_{j+1}\) corresponds to the vertex containing \(i_{j+1}\). Since
\(e\in E'\), the hyperedge cycle \(e\) is reversed in the partial dual, so on
this hyperedge we have
\[
\alpha^{E'}=\alpha^{-1}.
\]
In particular,
\[
\alpha^{E'}(i_{j+1})=\alpha^{-1}(i_{j+1})=i_j.
\]
Therefore, in \(H^{E'}\), the edge \(e(i_{j+1})\) connects the vertex
containing \(i_{j+1}\) to the vertex containing \(i_j\). These are exactly
the vertices corresponding to \(C_{j+1}\) and \(C_j\), respectively. Hence
\(C_j\) and \(C_{j+1}\) correspond to adjacent vertices, with the adjacency
induced by the dart \(i_{j+1}\). This establishes part~(2), and the final
assertion is then immediate.
\end{proof}

\begin{theorem}\label{thm:bip-char}
Let $H=(\sigma,\alpha)$ be an orientable hypermap, and let $E'\subseteq E(H)$. The following are equivalent:
\begin{enumerate}
  \item $H^{E'}$ is bipartite.
  \item There exists an all-crossing direction $\Phi$ of $M(H)$ such that
  \[
  E'=C(\Phi),
  \] where $C(\Phi)$ denotes the set of all type-$c$ hyperedges.
\end{enumerate}
\end{theorem}

\begin{proof}
\noindent(1)$\Rightarrow$(2). Assume that $H^{E'}$ is bipartite, and choose a bipartition
\[
V(H^{E'})=X\sqcup Y.
\]
By Theorem~\ref{thm:state-vertex}, each vertex corresponds to a unique state circle, so the state circles split into two classes: those corresponding to $X$, called $x$-circles, and those corresponding to $Y$, called $y$-circles. Every medial edge lies on a unique state circle, and therefore inherits a label $x$ or $y$.

Now orient each edge of the canonical checkerboard colored medial map $M(H)$ by the rule:
\begin{itemize}
  \item if the edge is labeled $x$, direct it so that the black face lies on its right;
  \item if the edge is labeled $y$, direct it so that the black face lies on its left.
\end{itemize}
This yields an orientation $\Phi$ satisfying
\[
\lambda(f)=x\Longleftrightarrow f\text{ lies on an }x\text{-circle},
\qquad
\lambda(f)=y\Longleftrightarrow f\text{ lies on a }y\text{-circle}.
\]

Take any hyperedge $e=(i_1\cdots i_k)$ and consider the corresponding medial vertex $v_e$.

\emph{Case A: $e\in E'$.} Then $v_e$ is white-smoothed. The two edges in each white smoothing block belong to the same state circle, hence they have the same $x/y$ label. On the other hand, by Lemma~\ref{lem:adjacent-blocks}(2), two adjacent white smoothing blocks correspond to two adjacent vertices of $H^{E'}$ along the hyperedge $e$; since $H^{E'}$ is bipartite, these two vertices must have opposite colors, and thus adjacent blocks must carry opposite labels. Therefore the labels on the white smoothing blocks around $v_e$ alternate cyclically as
\[
x,y,x,y,\dots.
\]
Hence the two edges in every white corner have the same label, whereas the two edges in every black corner have different labels. By Lemma~\ref{lem:xy-angle}, all white corners are crossing corners and all black corners are source/sink corners. Thus $v_e$ is a type-$c$ all-crossing vertex.

\emph{Case B: $e\notin E'$.} Then $v_e$ is black-smoothed. The same argument, using Lemma~\ref{lem:adjacent-blocks}(1), shows that the labels on the black smoothing blocks alternate, so every black corner is a crossing corner and every white corner is a source/sink corner. Hence $v_e$ is a type-$d$ all-crossing vertex.

Therefore $\Phi$ is an all-crossing direction, and
\[
e\in E'\Longleftrightarrow v_e\text{ is type-}c.
\]
So $E'=C(\Phi)$.

\noindent(2)$\Rightarrow$(1). Suppose that $M(H)$ admits an all-crossing direction $\Phi$ and $E'=C(\Phi)$. Let $\lambda$ be the induced $x/y$-labeling. Construct the state $S_{E'}$ by taking white smoothing for $e\in E'$ and black smoothing for $e\notin E'$. If $e\in E'$, then $v_e$ is a type-$c$ vertex, so all white corners are crossing corners; by Lemma~\ref{lem:xy-angle}, the two edges in each white smoothing block have the same label. If $e\notin E'$, then $v_e$ is a type-$d$ vertex, so all black corners are crossing corners; similarly, the two edges in each black smoothing block also have the same label. Thus, when one travels along any state circle, the $x/y$ label never changes at a smoothing block, and every state circle is monochromatic.

By Theorem~\ref{thm:state-vertex}, the state circles are in bijection with the vertices of $H^{E'}$. We may therefore color the vertices corresponding to $x$-circles with one color and those corresponding to $y$-circles with the other color.

It remains to verify that this is indeed a proper 2-coloring. Let $e$ be any hyperedge.
\begin{itemize}
  \item If $e\in E'$, then $v_e$ is type-$c$: the white smoothing blocks are monochromatic, while all black corners are source/sink corners. By Lemma~\ref{lem:xy-angle}, the two edges in each black corner have different labels, so adjacent white smoothing blocks must carry opposite labels. By Lemma~\ref{lem:adjacent-blocks}(2), these adjacent white smoothing blocks correspond to adjacent vertices in $H^{E'}$, and hence those vertices receive opposite colors.
  \item If $e\notin E'$, then $v_e$ is type-$d$: the black smoothing blocks are monochromatic, while all white corners are source/sink corners. Again by Lemma~\ref{lem:xy-angle}, adjacent black smoothing blocks must carry opposite labels. By Lemma~\ref{lem:adjacent-blocks}(1), the corresponding vertices are adjacent in $H^{E'}$, and therefore receive opposite colors as well.
\end{itemize}
Thus every edge connects vertices of opposite colors, and $H^{E'}$ is bipartite.
\end{proof}

\begin{remark}\label{rem:odd-from-proof}
The even-length obstruction in Corollary~\ref{cor:odd-obstruction} does not depend on Theorem~\ref{thm:bip-char}. Nevertheless, the proof above also reveals the same fact directly: at an all-crossing vertex, the labels on smoothing blocks must alternate as $x,y,x,y,\dots$ around the hyperedge, so the number of blocks must be even, and hence the corresponding hyperedge must have even length.
\end{remark}

\section{Further remarks and an example}\label{sec:further-remarks}

\begin{example}
We return to the hypermap $H_0=(\sigma,\alpha)$, where
\[
\sigma=(1\,8)(2\,10\,11)(3\,6)(4\,5\,7)(9\,12),
\qquad
\alpha=(1\,2\,3\,4)(5\,6\,7\,8)(9\,10)(11\,12).
\]
Write the four hyperedges as
\[
A:=(1,2,3,4),\qquad B:=(5,6,7,8),\qquad C:=(9,10),\qquad D:=(11,12).
\]
Then $A$ and $B$ are hyperedges of length $4$, while $C$ and $D$ have length $2$.

The corresponding vertices are
\[
(1,8),\qquad (2,10,11),\qquad (3,6),\qquad (4,5,7),\qquad (9,12).
\]

Using the partial-duality formula
\[
\sigma^{E'}=\alpha_{E'}^{-1}\sigma,
\]
one can compute directly the vertex-cycles of every partial dual $H_0^{E'}$ for $E'\subseteq\{A,B,C,D\}$. The calculation yields the following. Relevant figures can be found in the appendix.

\medskip
\noindent\textbf{(1) Eulerian partial duals.}
Exactly the following eight partial duals are Eulerian:
\[
H_0^{\{A\}},\quad
H_0^{\{A,B\}},\quad
H_0^{\{A,C\}},\quad
H_0^{\{A,D\}},\quad
H_0^{\{A,B,C\}},\quad
H_0^{\{A,B,D\}},\quad
H_0^{\{A,C,D\}},\quad
H_0^{\{A,B,C,D\}}.
\]
Equivalently,
\[
H_0^{E'} \text{ is Eulerian } \iff A\in E'.
\]
Hence, by Theorem~\ref{thm:eulerian-char}, for every subset $E'$ containing $A$ there exists a crossing-total direction $\Omega_{E'}$ of $M(H_0)$ such that
\[
E'=D(\Omega_{E'})\cup T',\qquad T'\subseteq T(\Omega_{E'}).
\]

\medskip
\noindent\textbf{(2) Bipartite partial duals.}
Exactly the following two partial duals are bipartite:
\[
H_0=H_0^{\varnothing},\qquad H_0^{\{C,D\}}.
\]
More explicitly:
\begin{itemize}
  \item when $E'=\varnothing$,
  \[
  \sigma^{\varnothing}=(1\,8)(2\,10\,11)(3\,6)(4\,5\,7)(9\,12),
  \]
  one may take the bipartition
  \[
  X=\{(1,8),(3,6),(9,12)\},
  \qquad
  Y=\{(2,10,11),(4,5,7)\};
  \]
  \item when $E'=\{C,D\}$,
  \[
  \sigma^{\{C,D\}}=(1\,8)(2\,9\,11)(3\,6)(4\,5\,7)(10\,12),
  \]
  one may take the bipartition
  \[
  X=\{(1,8),(3,6),(10,12)\},
  \qquad
  Y=\{(2,9,11),(4,5,7)\}.
  \]
\end{itemize}
Therefore, by Theorem~\ref{thm:bip-char}, there exist all-crossing directions $\Delta_0$ and $\Delta_1$ of $M(H_0)$ such that
\[
C(\Phi_0)=\varnothing,
\qquad
C(\Phi_1)=\{C,D\}.
\]

\medskip
\noindent\textbf{(3) Conclusion.}
Among the $16$ partial duals of this hypermap,
\[
\text{there are }8\text{ Eulerian partial duals},
\qquad
\text{and }2\text{ bipartite partial duals}.
\]
More precisely,
\[
H_0^{E'} \text{ is Eulerian } \iff A\in E',
\]
while
\[
H_0^{E'} \text{ is bipartite } \iff E'=\varnothing \text{ or } E'=\{C,D\}.
\]
This example shows that, in the hypermap setting, the distributions of Eulerian and bipartite partial duals can differ markedly from one another.
\end{example}

\begin{remark}\label{rem:ribbon}
When every hyperedge has length $2$, a hypermap reduces to a ribbon graph; in that case every medial vertex is $4$-valent, and all definitions and proofs in the present paper reduce to the results of Deng, Jin and Metsidik \cite{DengJinMetsidik2020}. Thus the present work may be viewed as a full extension of those two classical direction characterizations from ribbon graphs to orientable hypermaps with higher-valent medial vertices.
\end{remark}

\begin{remark}\label{rem:nonorientable}
We do not treat non-orientable hypermaps in this paper. The reason is that the $x/y$-labeling used in Theorem~\ref{thm:bip-char} depends on deciding whether the black face lies on the left or on the right of an edge, which requires a global orientation. By analogy with the non-orientable ribbon-graph case, one expects that some modified medial object would be needed in order to obtain a corresponding characterization \cite{DengJinMetsidik2020,EllisMonaghanMoffatt2013}.
\end{remark}

\begin{remark}\label{rem:literature}
The proof draws on three strands of literature: the basic theory and permutation models of hypermaps \cite{Walsh1975,JonesSingerman1978,Cori1975,GrossTucker1987,LandoZvonkin2004}; the extension of partial duality to hypermaps \cite{ChmutovVignes2022}; and the ribbon-graph characterizations of partial duals via medial-graph directions \cite{HuggettMoffatt2013,MetsidikJin2018,DengJinMetsidik2020}. To keep the paper self-contained, we import only those results that are genuinely needed and provide complete arguments again at the critical points.
\end{remark}

\section*{Declaration of competing interest}
We declare that we have no financial and personal relationships with other people or organizations that can inappropriately influence our work.

\section*{Acknowledgements}
This work is supported by Natural Science Foundation of Xinjiang (Grant Number: 2024D01A89) and National Natural Science Foundation of China (Grant Number: 11961070).

\section*{Data availability}
No data was used for the research described in the article.


\bibliographystyle{cas-model2-names}
\bibliography{main}

\newpage
\appendix
\section{Illustrations of Eulerian and bipartite partial duals}
\label{app:illustrations}
\subsection{All Eulerian partial duals: medial maps and partial dual diagrams}

The hypermap \(H_0\) has four hyperedges
\[
A=(1,2,3,4),\quad B=(5,6,7,8),\quad C=(9,10),\quad D=(11,12).
\]
By Theorem~5.3, crossing-total directions of \(M(H_0)\) detect Eulerian
partial duals through
\[
E'=D(\Omega)\cup T',\qquad T'\subseteq T(\Omega).
\]
In this example, the directions shown below give exactly the eight cases
\(A\in E'\). For each \(E'\), the left diagram is the corresponding
directed medial map, and the right diagram is the resulting Eulerian
partial dual \(H^{E'}_0\).

\medskip
\begin{figure}[h]
  \centering
  \begin{subfigure}[t]{0.351\textwidth}
    \centering
    \includegraphics[width=\textwidth]{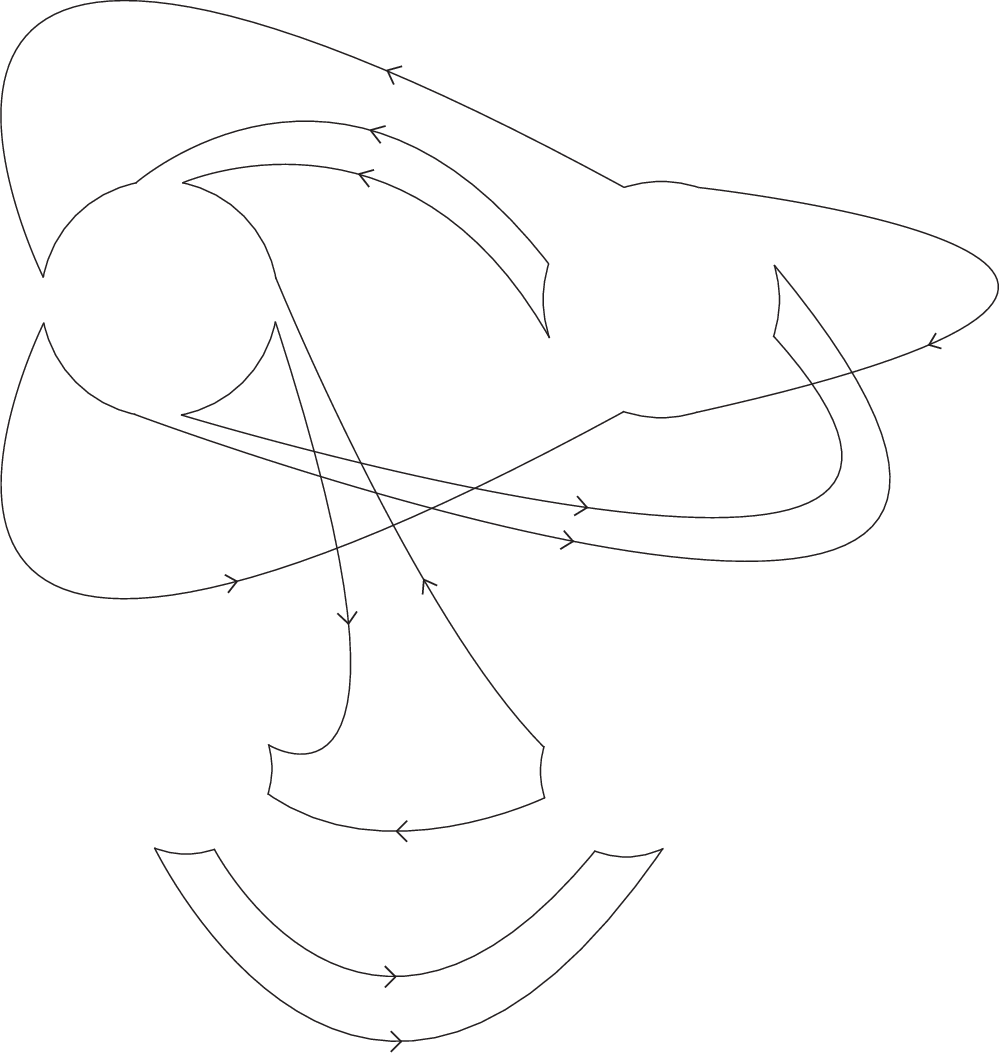}
    \caption{Medial map with crossing-total direction $\Omega_{\{A\}}$}
    \label{fig:eulerian_A_medial}
  \end{subfigure}
  \hspace{2cm}
  \begin{subfigure}[t]{0.39\textwidth}
    \centering
    \includegraphics[width=\textwidth]{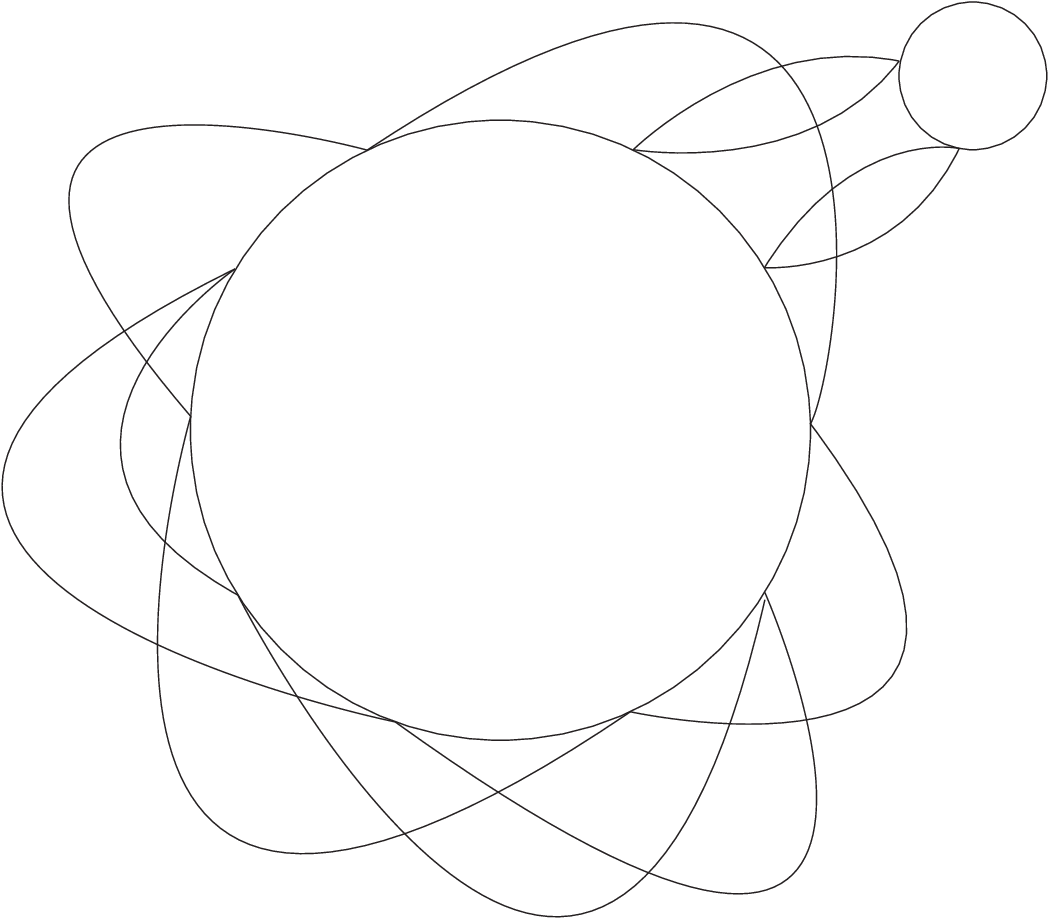}
    \caption{Partial dual $H_0^{\{A\}}$ }
    \label{fig:eulerian_A_dual}
  \end{subfigure}
\put(-395,134){\footnotesize$1^-$}
\put(-384,134){\footnotesize$1^+$}
\put(-376,125){\footnotesize$2^-$}
\put(-376,115){\footnotesize$2^+$}
\put(-383,107){\footnotesize$3^-$}
\put(-394,107){\footnotesize$3^+$}
\put(-401,116){\footnotesize$4^-$}
\put(-401,125){\footnotesize$4^+$}
\put(-310,134){\footnotesize$5^-$}
\put(-297,134){\footnotesize$5^+$}
\put(-292,126){\footnotesize$6^-$}
\put(-292,115){\footnotesize$6^+$}
\put(-297,107){\footnotesize$7^-$}
\put(-310,107){\footnotesize$7^+$}
\put(-319,116){\footnotesize$8^-$}
\put(-319,126){\footnotesize$8^+$}
\put(-378,34){\tiny$9^-$}
\put(-387,34){\tiny$9^+$}
\put(-380,50){\tiny$10^-$}
\put(-380,42){\tiny$10^+$}
\put(-321,42){\tiny$11^-$}
\put(-321,50){\tiny$11^+$}
\put(-303,35){\tiny$12^-$}
\put(-317,35){\tiny$12^+$}
\put(-120,125){\footnotesize1}
\put(-135,110){\footnotesize8}
\put(-143,86){\footnotesize4}
\put(-134,55){\footnotesize5}
\put(-113,36){\footnotesize7}
\put(-80,42){\footnotesize3}
\put(-60,55){\footnotesize6}
\put(-52,77){\footnotesize2}
\put(-62,106){\footnotesize10}
\put(-83,128){\footnotesize11}
\put(-14,137){\footnotesize9}
\put(-25,150){\footnotesize12}
  \caption{Eulerian case $E'=\{A\}$}
  \label{fig:eulerian_A}
\end{figure}

\begin{figure}[h]
  \centering
  \begin{subfigure}[t]{0.351\textwidth}
    \centering
    \includegraphics[width=\textwidth]{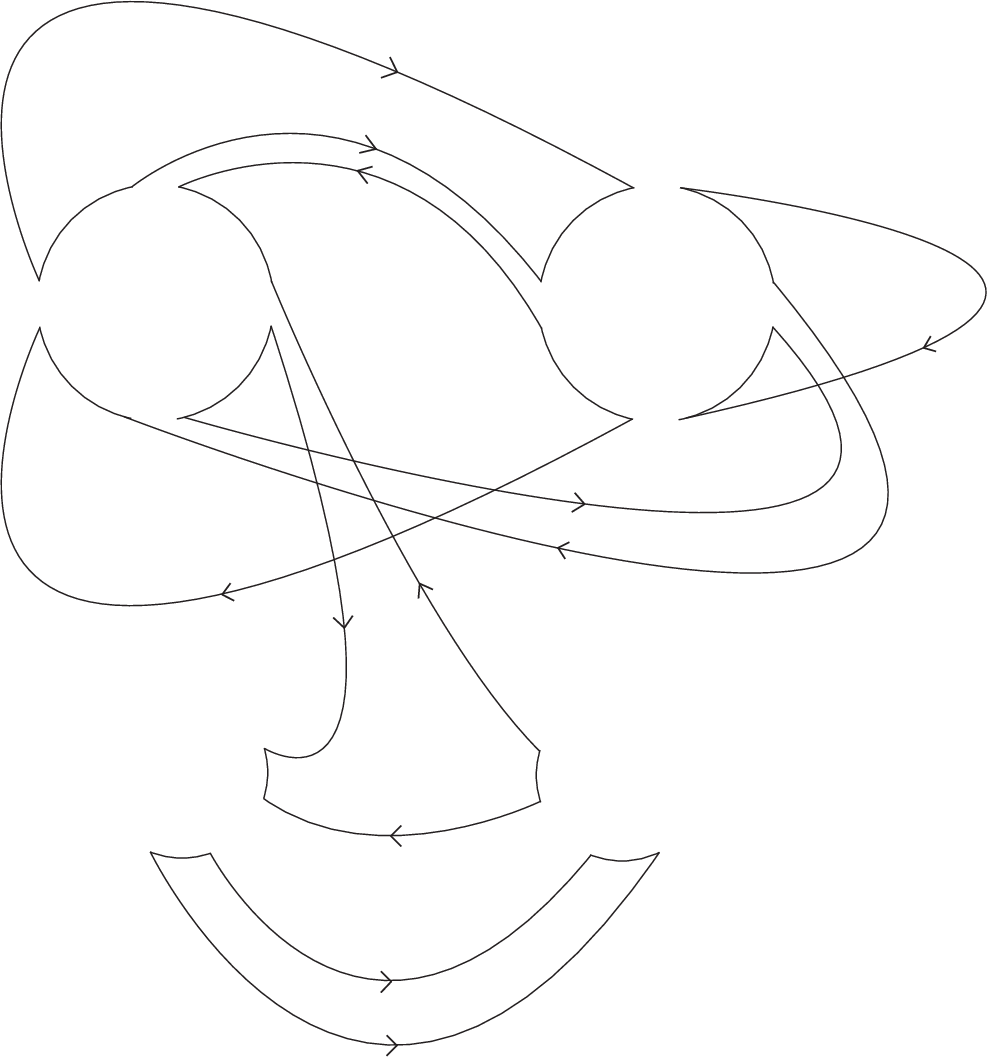}
    \caption{Medial map direction $\Omega_{\{A,B\}}$}
    \label{fig:eulerian_AB_medial}
  \end{subfigure}
    \hspace{2cm}
  \begin{subfigure}[t]{0.39\textwidth}
    \centering
    \includegraphics[width=\textwidth]{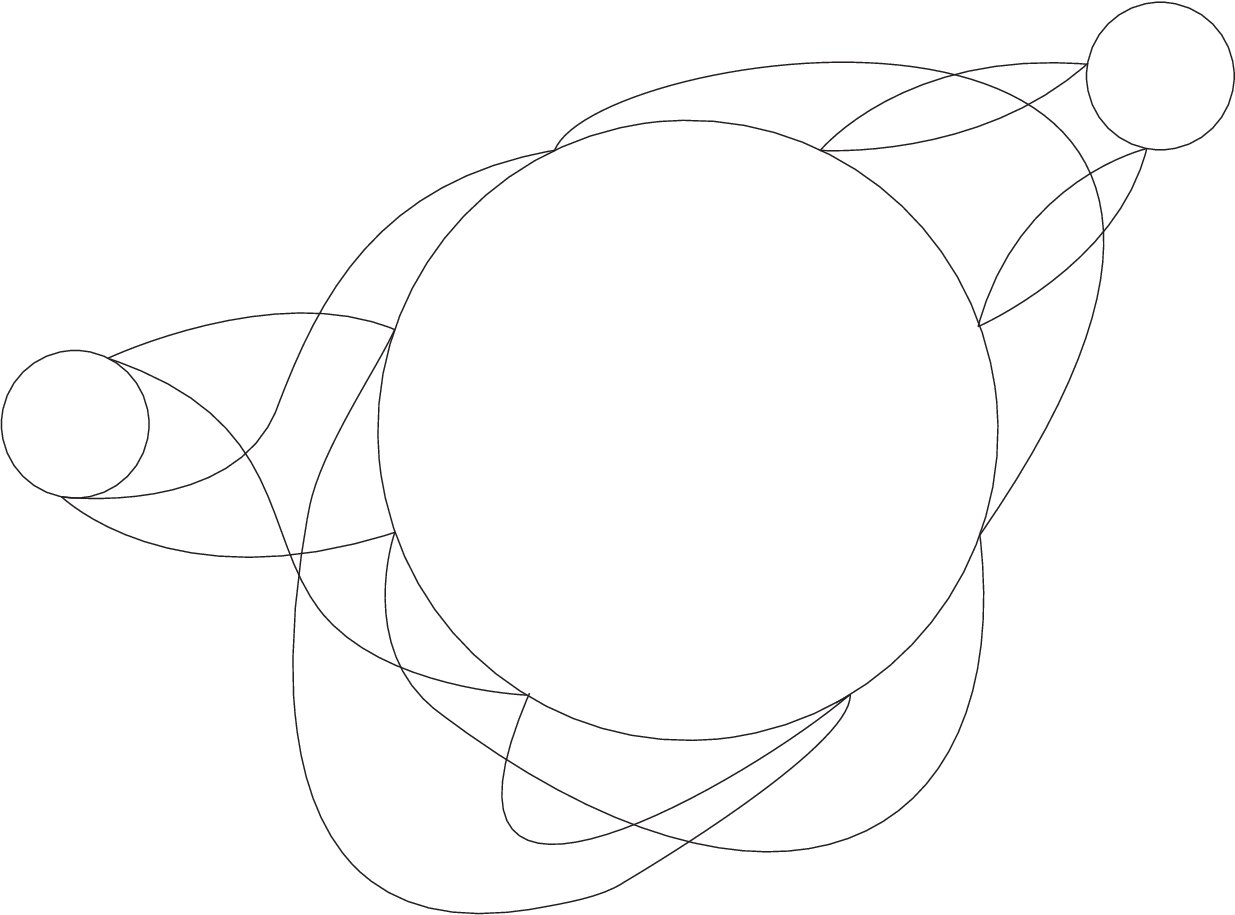}
    \caption{Partial dual $H_0^{\{A,B\}}$}
    \label{fig:eulerian_AB_dual}
  \end{subfigure}
\put(-395,136){\footnotesize$1^-$}
\put(-384,136){\footnotesize$1^+$}
\put(-376,125){\footnotesize$2^-$}
\put(-376,115){\footnotesize$2^+$}
\put(-383,107){\footnotesize$3^-$}
\put(-394,107){\footnotesize$3^+$}
\put(-401,116){\footnotesize$4^-$}
\put(-401,125){\footnotesize$4^+$}
\put(-310,136){\footnotesize$5^-$}
\put(-297,136){\footnotesize$5^+$}
\put(-292,126){\footnotesize$6^-$}
\put(-292,115){\footnotesize$6^+$}
\put(-297,107){\footnotesize$7^-$}
\put(-310,107){\footnotesize$7^+$}
\put(-319,116){\footnotesize$8^-$}
\put(-319,126){\footnotesize$8^+$}
\put(-378,34){\tiny$9^-$}
\put(-387,34){\tiny$9^+$}
\put(-380,50){\tiny$10^-$}
\put(-380,42){\tiny$10^+$}
\put(-319,42){\tiny$11^-$}
\put(-319,50){\tiny$11^+$}
\put(-303,35){\tiny$12^-$}
\put(-317,35){\tiny$12^+$}
\put(-100,104){\footnotesize1}
\put(-119,84){\footnotesize7}
\put(-120,51){\footnotesize3}
\put(-103,35){\footnotesize5}
\put(-64,34){\footnotesize6}
\put(-175,62){\footnotesize4}
\put(-47,56){\footnotesize2}
\put(-52,86){\footnotesize10}
\put(-70,105){\footnotesize11}
\put(-170,74){\footnotesize8}
\put(-10,117){\footnotesize9}
\put(-22,124){\footnotesize12}
  \caption{Eulerian case $E'=\{A,B\}$}
  \label{fig:eulerian_AB}
\end{figure}

\begin{figure}[h]
  \centering
  \begin{subfigure}[t]{0.351\textwidth}
    \centering
    \includegraphics[width=\textwidth]{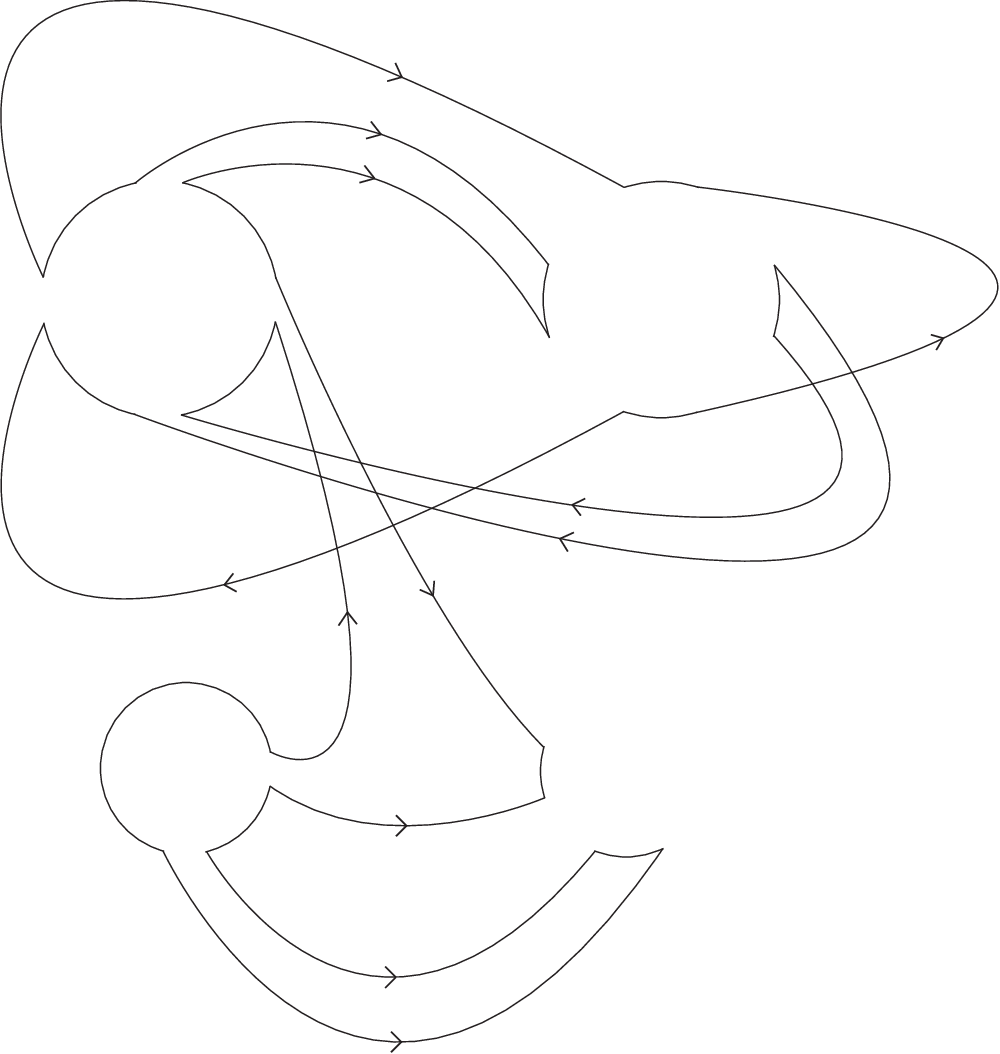}
    \caption{Medial map direction $\Omega_{\{A,C\}}$}
    \label{fig:eulerian_AC_medial}
  \end{subfigure}
  \hspace{2cm}
  \begin{subfigure}[t]{0.39\textwidth}
    \centering
    \includegraphics[width=\textwidth]{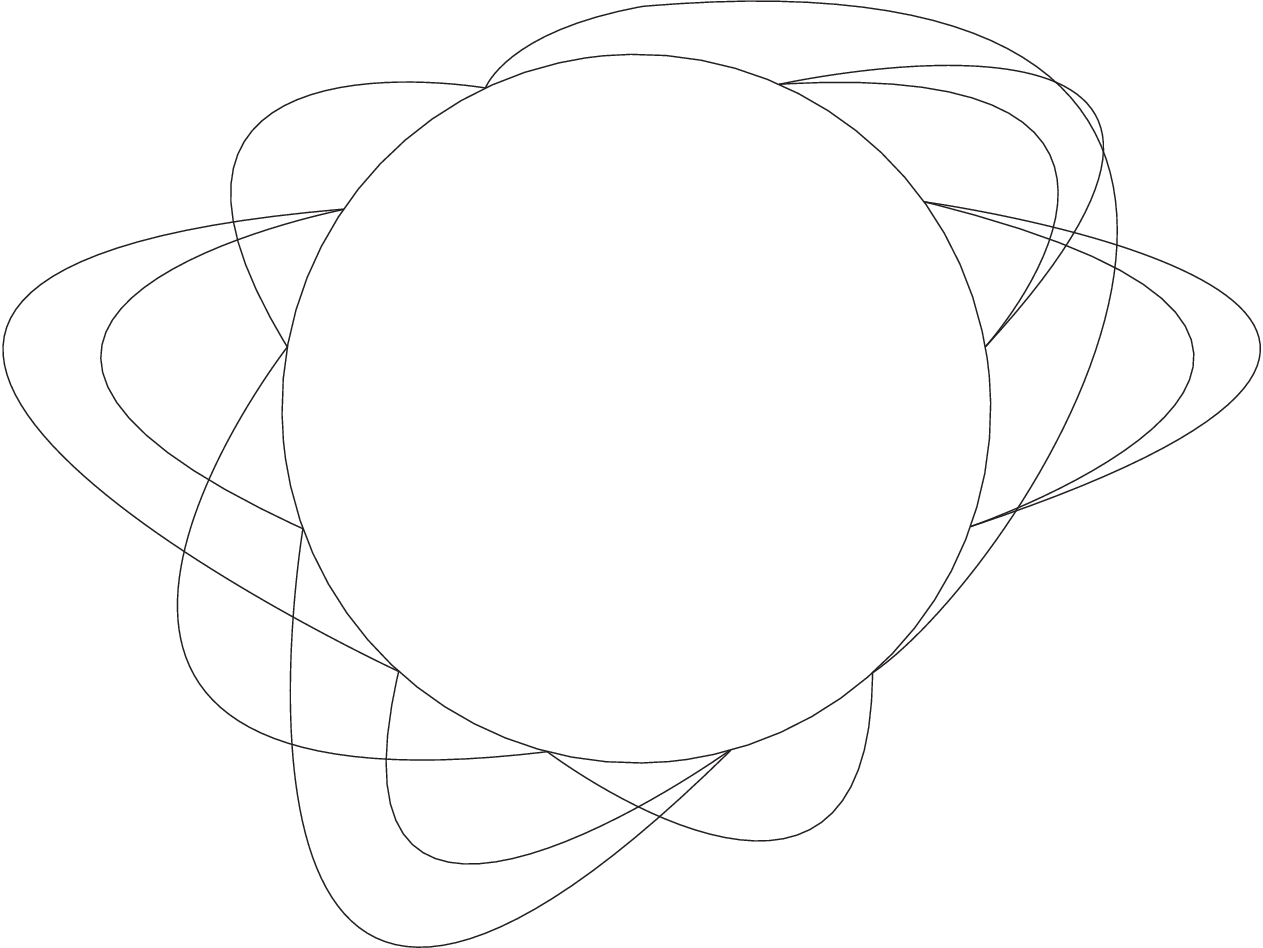}
    \caption{Partial dual $H_0^{\{A,C\}}$}
    \label{fig:eulerian_AC_dual}
  \end{subfigure}
\put(-395,136){\footnotesize$1^-$}
\put(-384,136){\footnotesize$1^+$}
\put(-376,125){\footnotesize$2^-$}
\put(-376,115){\footnotesize$2^+$}
\put(-383,107){\footnotesize$3^-$}
\put(-394,107){\footnotesize$3^+$}
\put(-401,116){\footnotesize$4^-$}
\put(-401,125){\footnotesize$4^+$}
\put(-310,136){\footnotesize$5^-$}
\put(-297,136){\footnotesize$5^+$}
\put(-292,126){\footnotesize$6^-$}
\put(-292,115){\footnotesize$6^+$}
\put(-297,107){\footnotesize$7^-$}
\put(-310,107){\footnotesize$7^+$}
\put(-319,116){\footnotesize$8^-$}
\put(-319,126){\footnotesize$8^+$}
\put(-378,34){\tiny$9^-$}
\put(-387,34){\tiny$9^+$}
\put(-380,50){\tiny$10^-$}
\put(-380,42){\tiny$10^+$}
\put(-319,42){\tiny$11^-$}
\put(-319,50){\tiny$11^+$}
\put(-303,35){\tiny$12^-$}
\put(-317,35){\tiny$12^+$}
\put(-116,113){\footnotesize1}
\put(-131,100){\footnotesize8}
\put(-138,82){\footnotesize4}
\put(-134,59){\footnotesize5}
\put(-119,41){\footnotesize7}
\put(-100,32){\footnotesize3}
\put(-80,32){\footnotesize6}
\put(-64,40){\footnotesize2}
\put(-62,102){\footnotesize10}
\put(-80,114){\footnotesize11}
\put(-55,56){\footnotesize9}
\put(-55,80){\footnotesize12}
  \caption{Eulerian case $E'=\{A,C\}$}
  \label{fig:eulerian_AC}
\end{figure}

\begin{figure}[h]
  \centering
  \begin{subfigure}[t]{0.351\textwidth}
    \centering
    \includegraphics[width=\textwidth]{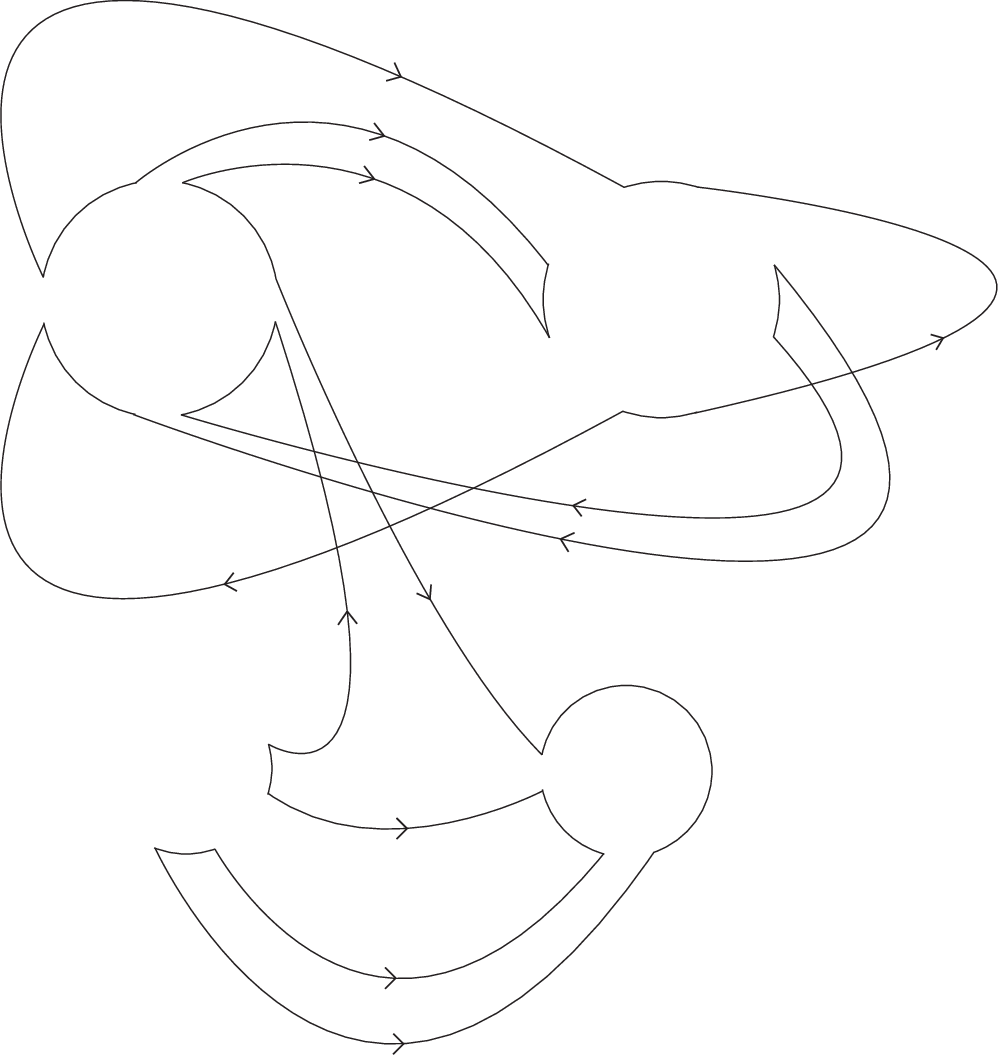}
    \caption{Medial map direction $\Omega_{\{A,D\}}$}
    \label{fig:eulerian_AD_medial}
  \end{subfigure}
  \hspace{2cm}
  \begin{subfigure}[t]{0.39\textwidth}
    \centering
    \includegraphics[width=\textwidth]{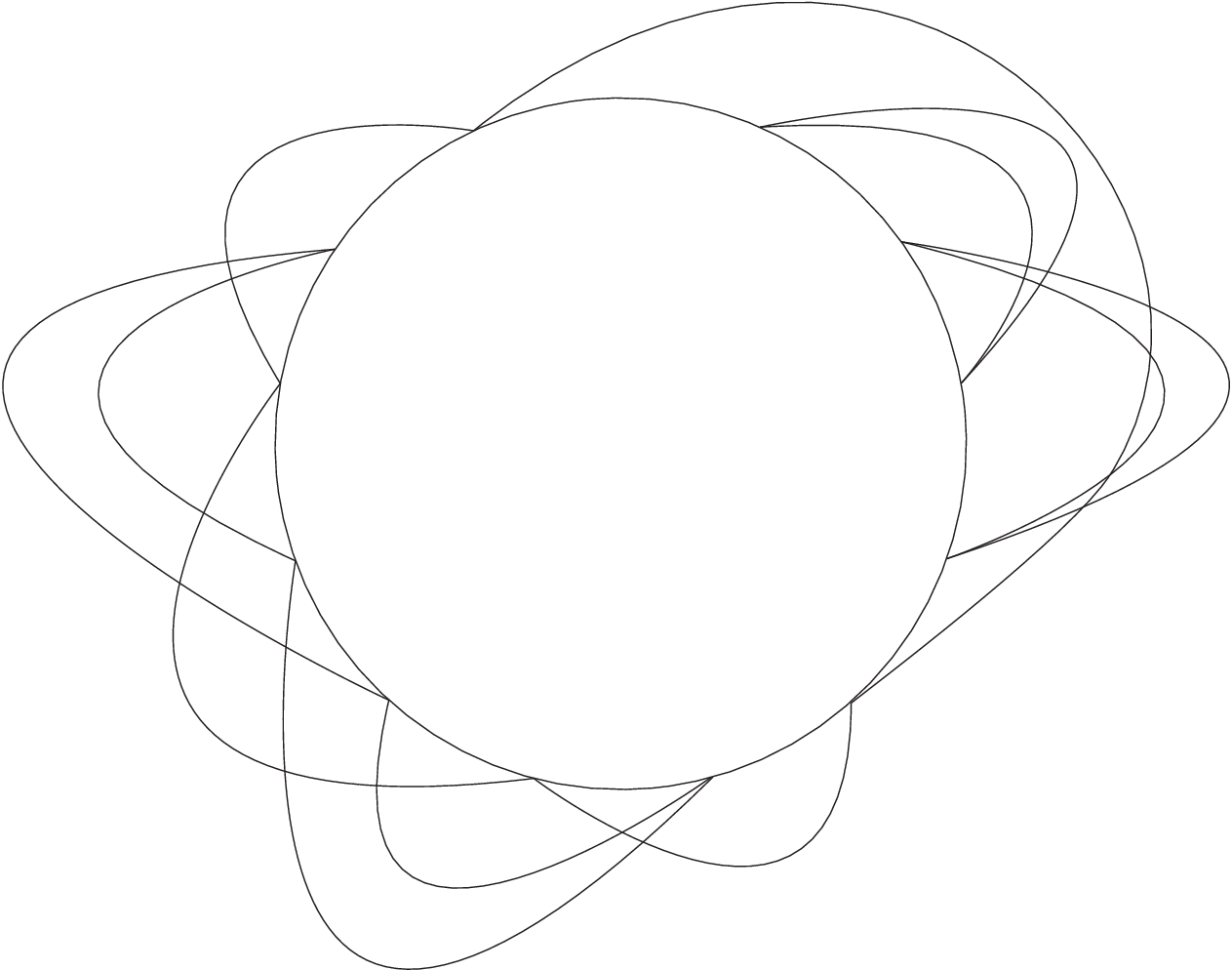}
    \caption{Partial dual $H_0^{\{A,D\}}$}
    \label{fig:eulerian_AD_dual}
  \end{subfigure}
\put(-395,136){\footnotesize$1^-$}
\put(-384,136){\footnotesize$1^+$}
\put(-376,125){\footnotesize$2^-$}
\put(-376,115){\footnotesize$2^+$}
\put(-383,107){\footnotesize$3^-$}
\put(-394,107){\footnotesize$3^+$}
\put(-401,116){\footnotesize$4^-$}
\put(-401,125){\footnotesize$4^+$}
\put(-310,136){\footnotesize$5^-$}
\put(-297,136){\footnotesize$5^+$}
\put(-292,126){\footnotesize$6^-$}
\put(-292,115){\footnotesize$6^+$}
\put(-297,107){\footnotesize$7^-$}
\put(-310,107){\footnotesize$7^+$}
\put(-319,116){\footnotesize$8^-$}
\put(-319,126){\footnotesize$8^+$}
\put(-378,34){\tiny$9^-$}
\put(-387,34){\tiny$9^+$}
\put(-380,50){\tiny$10^-$}
\put(-380,42){\tiny$10^+$}
\put(-319,42){\tiny$11^-$}
\put(-319,50){\tiny$11^+$}
\put(-303,35){\tiny$12^-$}
\put(-317,35){\tiny$12^+$}
\put(-116,113){\footnotesize1}
\put(-131,100){\footnotesize8}
\put(-138,82){\footnotesize4}
\put(-134,59){\footnotesize5}
\put(-119,41){\footnotesize7}
\put(-100,32){\footnotesize3}
\put(-80,32){\footnotesize6}
\put(-64,40){\footnotesize2}
\put(-55,56){\footnotesize10}
\put(-55,80){\footnotesize12}
\put(-77,116){\footnotesize9}
\put(-62,102){\footnotesize11}
  \caption{Eulerian case $E'=\{A,D\}$}
  \label{fig:eulerian_AD}
\end{figure}

\begin{figure}[h]
  \centering
  \begin{subfigure}[t]{0.350\textwidth}
    \centering
    \includegraphics[width=\textwidth]{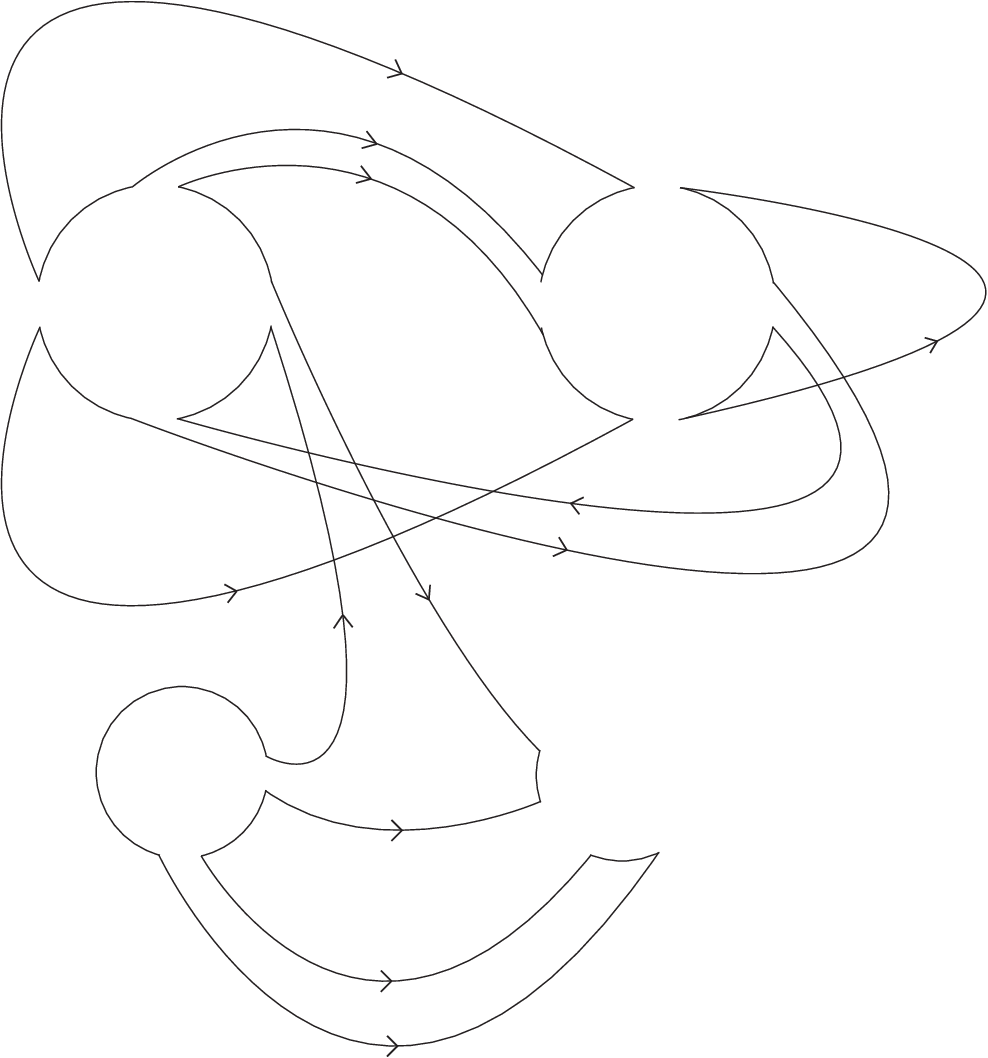}
    \caption{Medial map direction $\Omega_{\{A,B,C\}}$}
    \label{fig:eulerian_ABC_medial}
  \end{subfigure}
  \hspace{2cm}
  \begin{subfigure}[t]{0.39\textwidth}
    \centering
    \includegraphics[width=\textwidth]{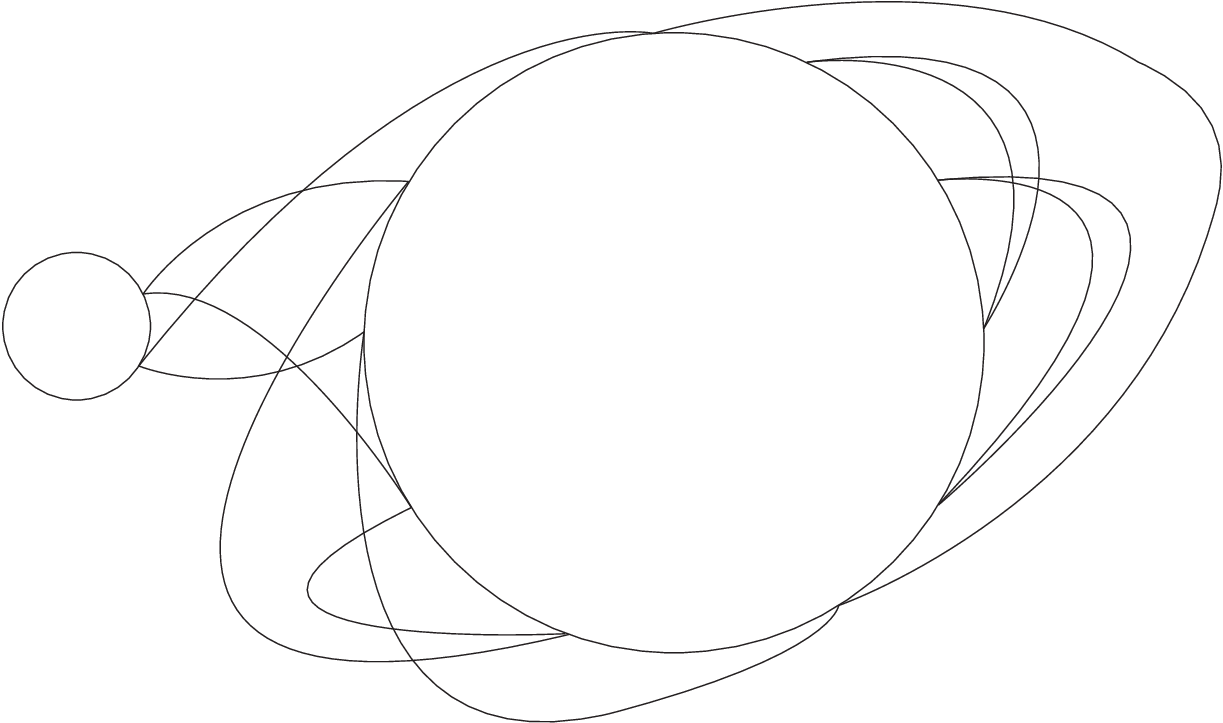}
    \caption{Partial dual $H_0^{\{A,B,C\}}$}
    \label{fig:eulerian_ABC_dual}
  \end{subfigure}
\put(-395,136){\footnotesize$1^-$}
\put(-384,136){\footnotesize$1^+$}
\put(-376,125){\footnotesize$2^-$}
\put(-376,115){\footnotesize$2^+$}
\put(-383,107){\footnotesize$3^-$}
\put(-394,107){\footnotesize$3^+$}
\put(-401,116){\footnotesize$4^-$}
\put(-401,125){\footnotesize$4^+$}
\put(-310,136){\footnotesize$5^-$}
\put(-297,136){\footnotesize$5^+$}
\put(-292,126){\footnotesize$6^-$}
\put(-292,115){\footnotesize$6^+$}
\put(-297,107){\footnotesize$7^-$}
\put(-310,107){\footnotesize$7^+$}
\put(-319,116){\footnotesize$8^-$}
\put(-319,126){\footnotesize$8^+$}
\put(-378,34){\tiny$9^-$}
\put(-387,34){\tiny$9^+$}
\put(-380,50){\tiny$10^-$}
\put(-380,42){\tiny$10^+$}
\put(-319,42){\tiny$11^-$}
\put(-319,50){\tiny$11^+$}
\put(-303,35){\tiny$12^-$}
\put(-317,35){\tiny$12^+$}
\put(-90,92){\footnotesize1}
\put(-120,75){\footnotesize7}
\put(-123,56){\footnotesize3}
\put(-116,32){\footnotesize5}
\put(-97,18){\footnotesize6}
\put(-66,18){\footnotesize2}
\put(-50,30){\footnotesize9}
\put(-52,56){\footnotesize12}
\put(-55,75){\footnotesize10}
\put(-67,88){\footnotesize11}
\put(-170,50){\footnotesize4}
\put(-170,62){\footnotesize8}
  \caption{Eulerian case $E'=\{A,B,C\}$}
  \label{fig:eulerian_ABC}
\end{figure}

\begin{figure}[h]
  \centering
  \begin{subfigure}[t]{0.347\textwidth}
    \centering
    \includegraphics[width=\textwidth]{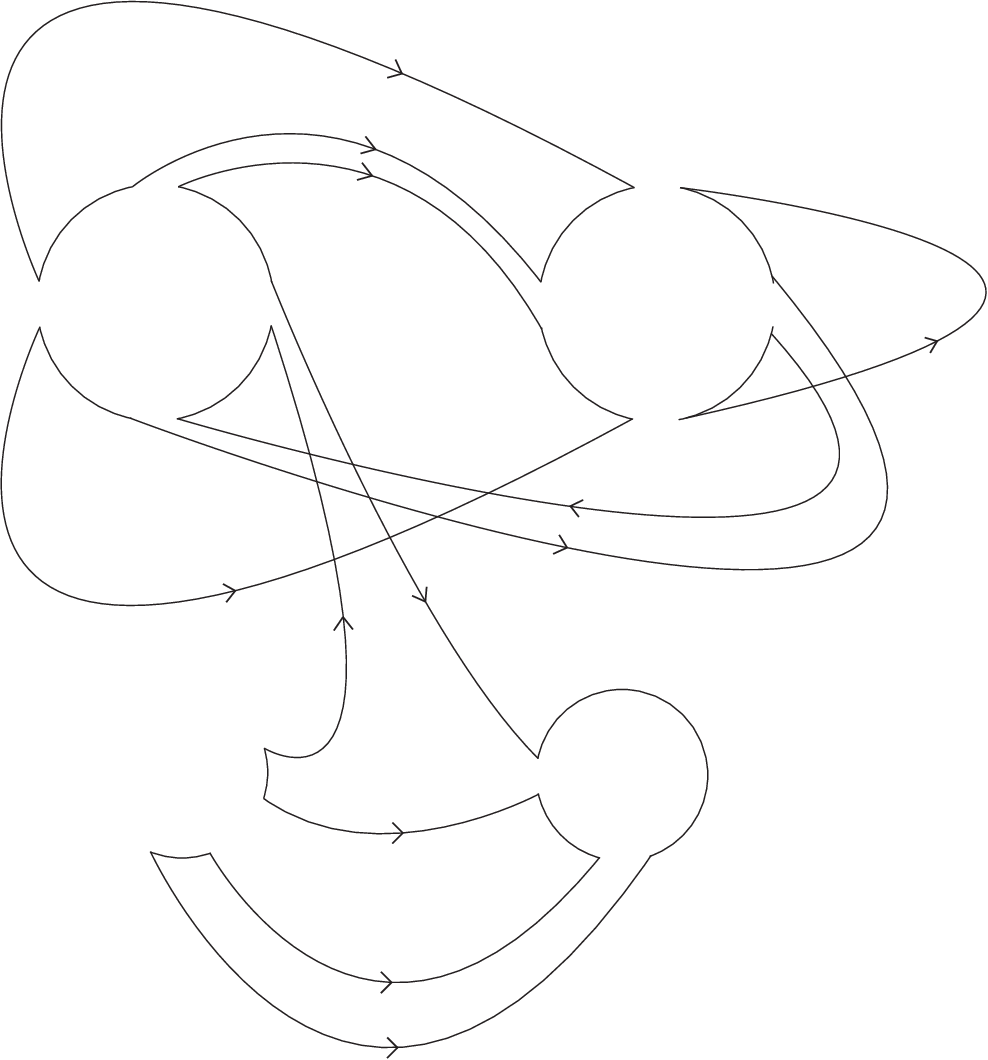}
    \caption{Medial map direction $\Omega_{\{A,B,D\}}$}
    \label{fig:eulerian_ABD_medial}
  \end{subfigure}
  \hspace{2cm}
  \begin{subfigure}[t]{0.39\textwidth}
    \centering
    \includegraphics[width=\textwidth]{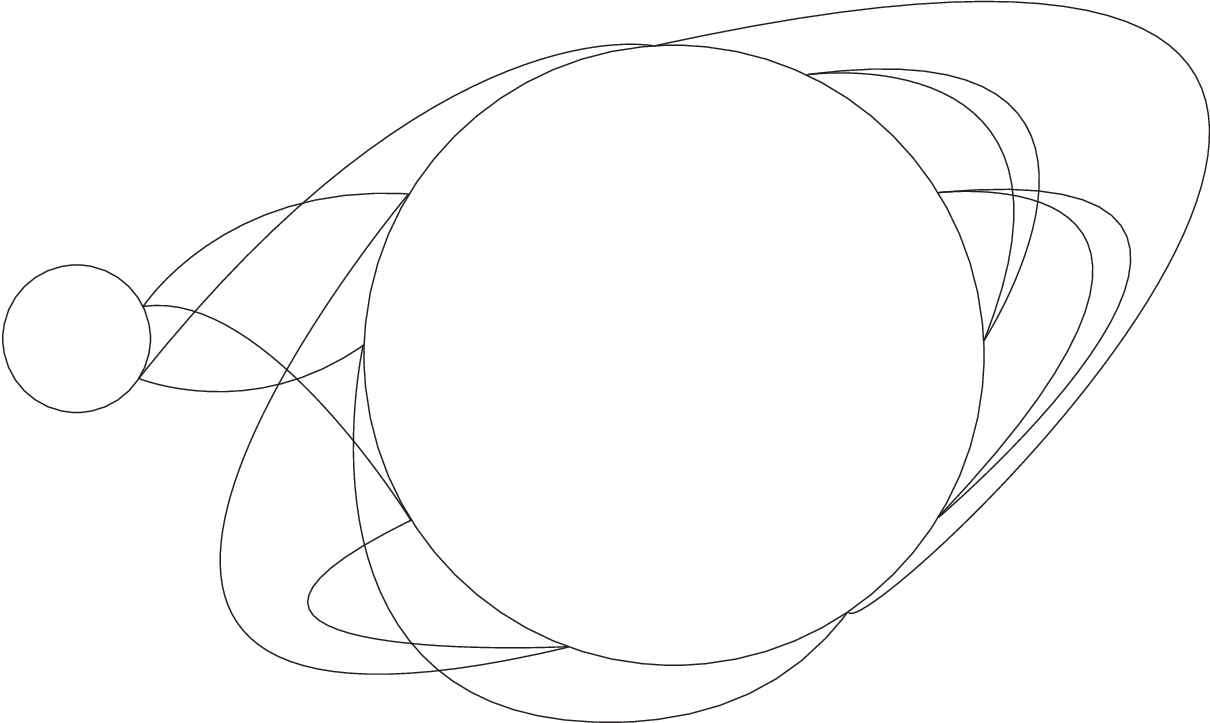}
    \caption{Partial dual $H_0^{\{A,B,D\}}$}
    \label{fig:eulerian_ABD_dual}
  \end{subfigure}
\put(-393,136){\footnotesize$1^-$}
\put(-382,136){\footnotesize$1^+$}
\put(-376,125){\footnotesize$2^-$}
\put(-376,115){\footnotesize$2^+$}
\put(-383,107){\footnotesize$3^-$}
\put(-394,107){\footnotesize$3^+$}
\put(-401,116){\footnotesize$4^-$}
\put(-401,125){\footnotesize$4^+$}
\put(-310,136){\footnotesize$5^-$}
\put(-297,136){\footnotesize$5^+$}
\put(-292,126){\footnotesize$6^-$}
\put(-292,115){\footnotesize$6^+$}
\put(-297,107){\footnotesize$7^-$}
\put(-310,107){\footnotesize$7^+$}
\put(-319,116){\footnotesize$8^-$}
\put(-319,126){\footnotesize$8^+$}
\put(-378,34){\tiny$9^-$}
\put(-387,34){\tiny$9^+$}
\put(-378,50){\tiny$10^-$}
\put(-378,42){\tiny$10^+$}
\put(-319,42){\tiny$11^-$}
\put(-319,50){\tiny$11^+$}
\put(-303,35){\tiny$12^-$}
\put(-317,35){\tiny$12^+$}
\put(-90,92){\footnotesize1}
\put(-120,75){\footnotesize7}
\put(-123,56){\footnotesize3}
\put(-119,32){\footnotesize5}
\put(-100,16){\footnotesize6}
\put(-62,18){\footnotesize2}
\put(-53,30){\footnotesize10}
\put(-51,56){\footnotesize12}
\put(-52,75){\footnotesize9}
\put(-67,88){\footnotesize11}
\put(-170,49){\footnotesize4}
\put(-170,61){\footnotesize8}
  \caption{Eulerian case $E'=\{A,B,D\}$}
  \label{fig:eulerian_ABD}
\end{figure}

\begin{figure}[h]
  \centering
  \begin{subfigure}[t]{0.347\textwidth}
    \centering
    \includegraphics[width=\textwidth]{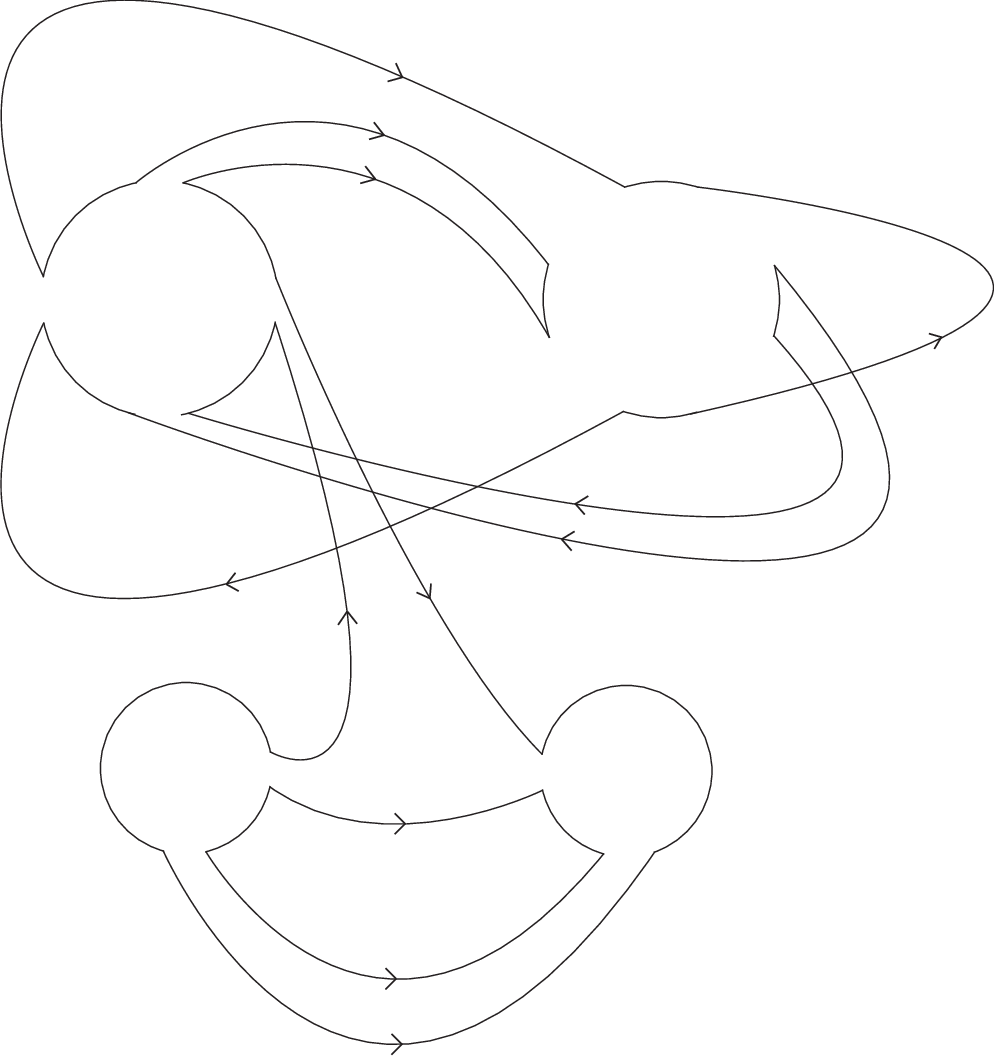}
    \caption{Medial map direction $\Omega_{\{A,C,D\}}$}
    \label{fig:eulerian_ACD_medial}
  \end{subfigure}
  \hspace{2cm}
  \begin{subfigure}[t]{0.39\textwidth}
    \centering
    \includegraphics[width=\textwidth]{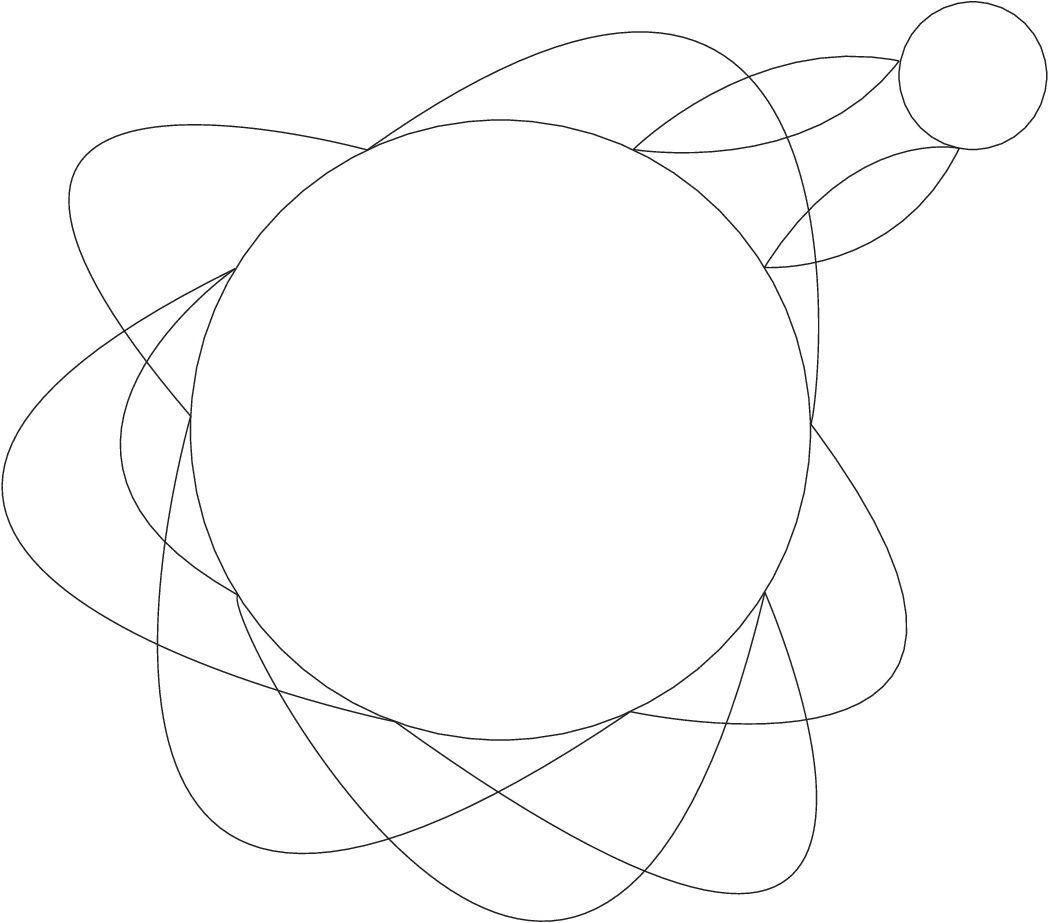}
    \caption{Partial dual $H_0^{\{A,C,D\}}$}
    \label{fig:eulerian_ACD_dual}
  \end{subfigure}
\put(-393,136){\footnotesize$1^-$}
\put(-382,136){\footnotesize$1^+$}
\put(-376,125){\footnotesize$2^-$}
\put(-376,115){\footnotesize$2^+$}
\put(-383,107){\footnotesize$3^-$}
\put(-394,107){\footnotesize$3^+$}
\put(-401,116){\footnotesize$4^-$}
\put(-401,125){\footnotesize$4^+$}
\put(-310,134){\footnotesize$5^-$}
\put(-297,134){\footnotesize$5^+$}
\put(-292,126){\footnotesize$6^-$}
\put(-292,115){\footnotesize$6^+$}
\put(-297,107){\footnotesize$7^-$}
\put(-310,107){\footnotesize$7^+$}
\put(-319,116){\footnotesize$8^-$}
\put(-319,126){\footnotesize$8^+$}
\put(-378,34){\tiny$9^-$}
\put(-387,34){\tiny$9^+$}
\put(-378,50){\tiny$10^-$}
\put(-378,42){\tiny$10^+$}
\put(-319,42){\tiny$11^-$}
\put(-319,50){\tiny$11^+$}
\put(-303,35){\tiny$12^-$}
\put(-317,35){\tiny$12^+$}
\put(-120,125){\footnotesize1}
\put(-135,110){\footnotesize8}
\put(-145,86){\footnotesize4}
\put(-133,55){\footnotesize5}
\put(-113,36){\footnotesize7}
\put(-80,39){\footnotesize3}
\put(-60,55){\footnotesize6}
\put(-50,79){\footnotesize2}
\put(-62,106){\footnotesize9}
\put(-83,125){\footnotesize11}
\put(-14,137){\footnotesize10}
\put(-25,147){\footnotesize12}
  \caption{Eulerian case $E'=\{A,C,D\}$}
  \label{fig:eulerian_ACD}
\end{figure}

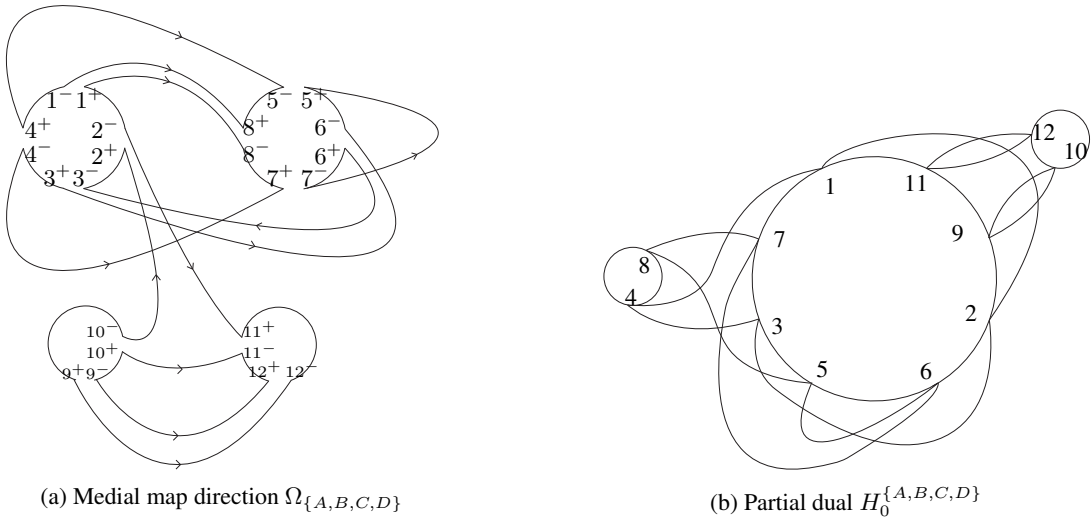
\begin{figure}[h]
  \centering
  \begin{subfigure}[t]{0.347\textwidth}
    \centering
    \includegraphics[width=\textwidth]{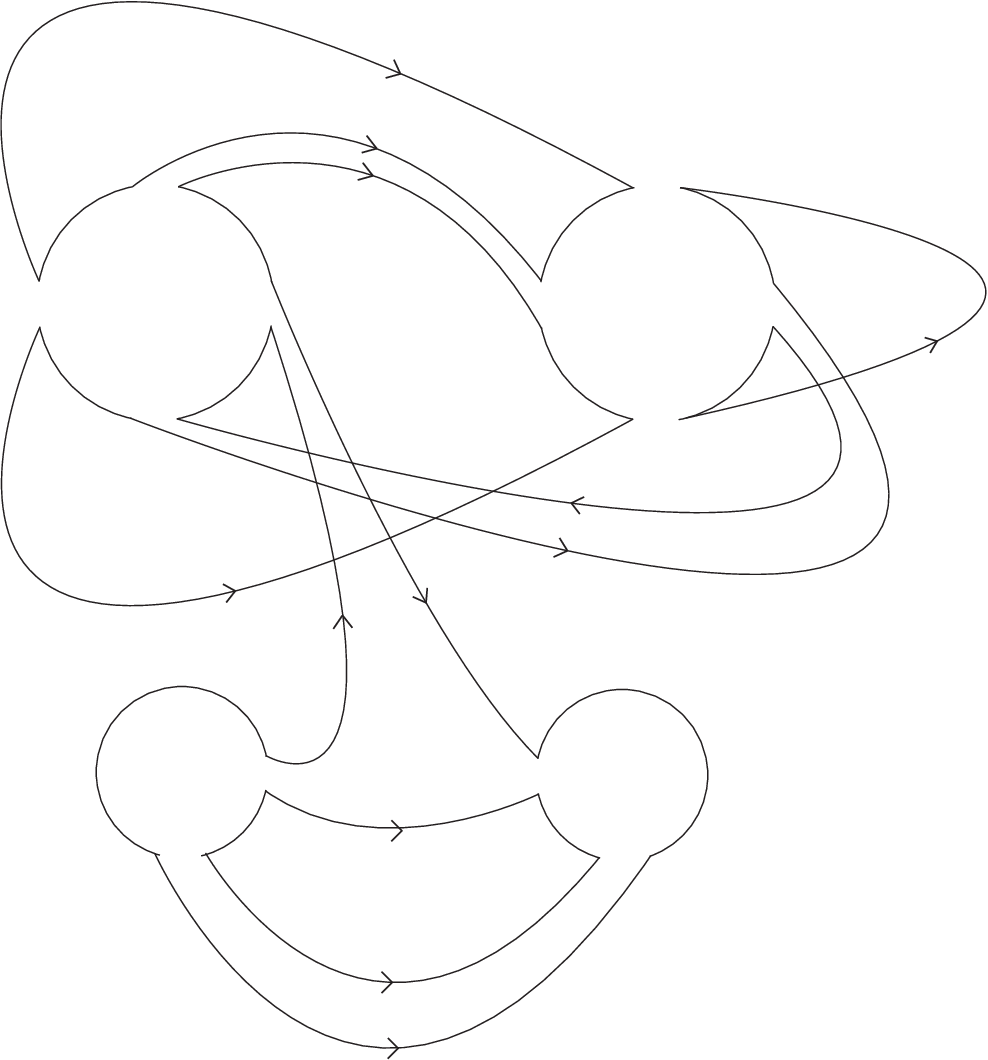}
    \caption{Medial map direction $\Omega_{\{A,B,C,D\}}$}
    \label{fig:eulerian_ABCD_medial}
  \end{subfigure}
  \hspace{2cm}
  \begin{subfigure}[t]{0.39\textwidth}
    \centering
    \includegraphics[width=\textwidth]{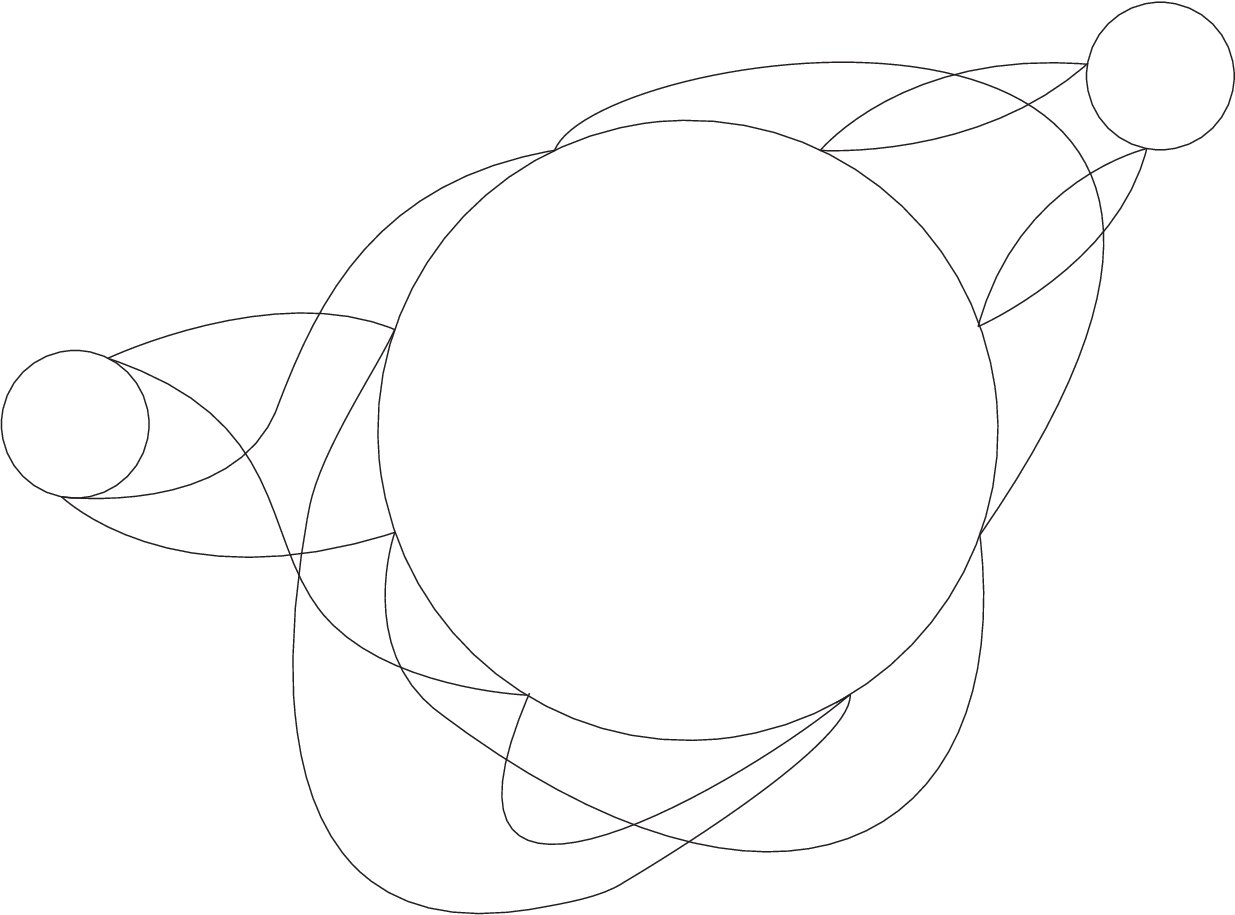}
    \caption{Partial dual $H_0^{\{A,B,C,D\}}$}
    \label{fig:eulerian_ABCD_dual}
  \end{subfigure}
\put(-393,136){\footnotesize$1^-$}
\put(-382,136){\footnotesize$1^+$}
\put(-376,125){\footnotesize$2^-$}
\put(-376,115){\footnotesize$2^+$}
\put(-383,107){\footnotesize$3^-$}
\put(-394,107){\footnotesize$3^+$}
\put(-401,116){\footnotesize$4^-$}
\put(-401,125){\footnotesize$4^+$}
\put(-310,136){\footnotesize$5^-$}
\put(-297,136){\footnotesize$5^+$}
\put(-292,126){\footnotesize$6^-$}
\put(-292,115){\footnotesize$6^+$}
\put(-297,107){\footnotesize$7^-$}
\put(-310,107){\footnotesize$7^+$}
\put(-319,116){\footnotesize$8^-$}
\put(-319,126){\footnotesize$8^+$}
\put(-378,34){\tiny$9^-$}
\put(-387,34){\tiny$9^+$}
\put(-378,50){\tiny$10^-$}
\put(-378,42){\tiny$10^+$}
\put(-319,42){\tiny$11^-$}
\put(-319,50){\tiny$11^+$}
\put(-303,35){\tiny$12^-$}
\put(-317,35){\tiny$12^+$}
\put(-100,104){\footnotesize1}
\put(-119,84){\footnotesize7}
\put(-120,51){\footnotesize3}
\put(-103,35){\footnotesize5}
\put(-64,34){\footnotesize6}
\put(-175,62){\footnotesize4}
\put(-47,56){\footnotesize2}
\put(-52,86){\footnotesize9}
\put(-70,105){\footnotesize11}
\put(-170,74){\footnotesize8}
\put(-10,117){\footnotesize10}
\put(-22,124){\footnotesize12}
  \caption{Eulerian case $E'=\{A,B,C,D\}$}
  \label{fig:eulerian_ABCD}
\end{figure}

\clearpage

\subsection{All bipartite partial duals: medial maps and partial dual diagrams}

The same hypermap $H_0$ admits exactly two bipartite partial duals: $E'=\varnothing$ (the original hypermap) and $E'=\{C,D\}$. Below are the medial maps with all-crossing directions $\phi_0,\phi_1$, together with the corresponding bipartite partial duals.

\begin{figure}[h]
  \centering
  \begin{subfigure}[t]{0.351\textwidth}
    \centering
    \includegraphics[width=\textwidth]{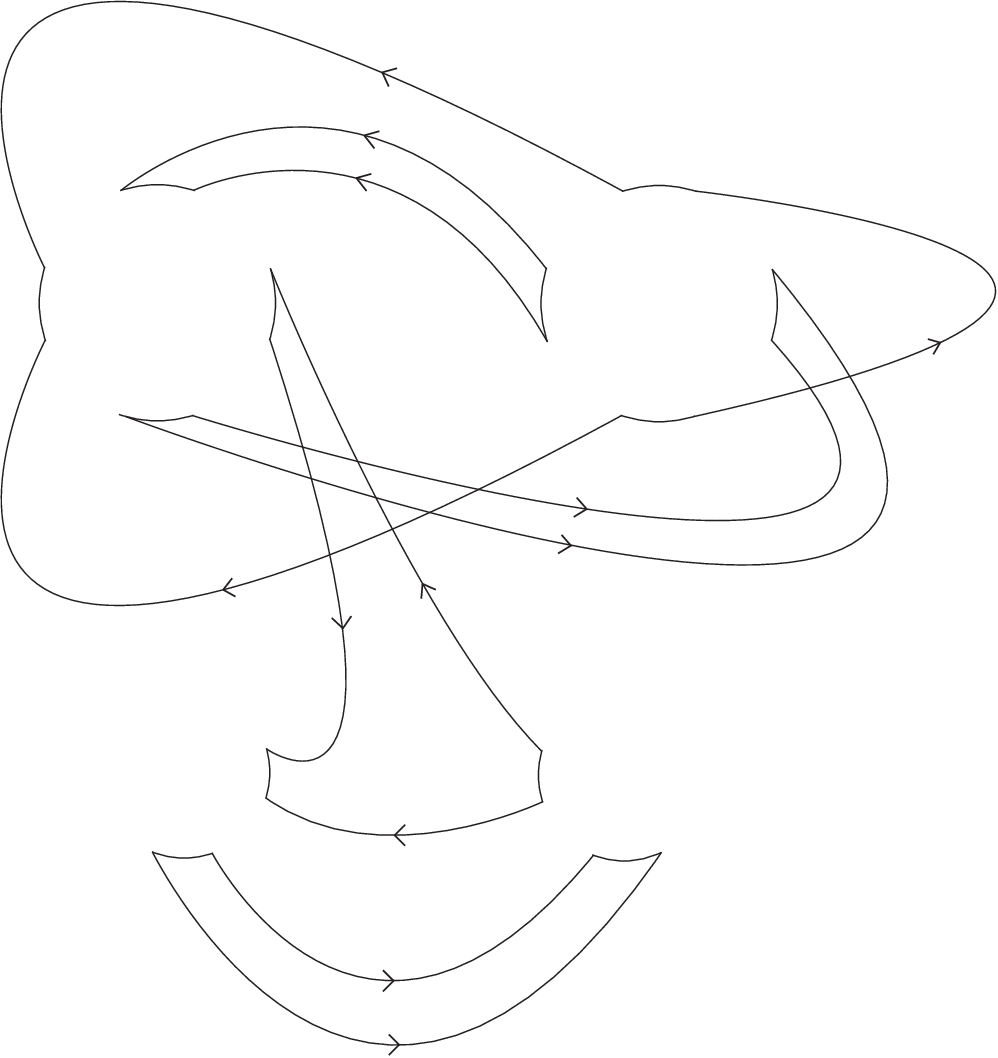}
    \caption{Medial map with all-crossing direction $\Phi_0$ ($C(\Phi_0)=\varnothing$)}
    \label{fig:bipartite_empty_medial}
  \end{subfigure}
  \hspace{2cm}
  \begin{subfigure}[t]{0.39\textwidth}
    \centering
    \includegraphics[width=\textwidth]{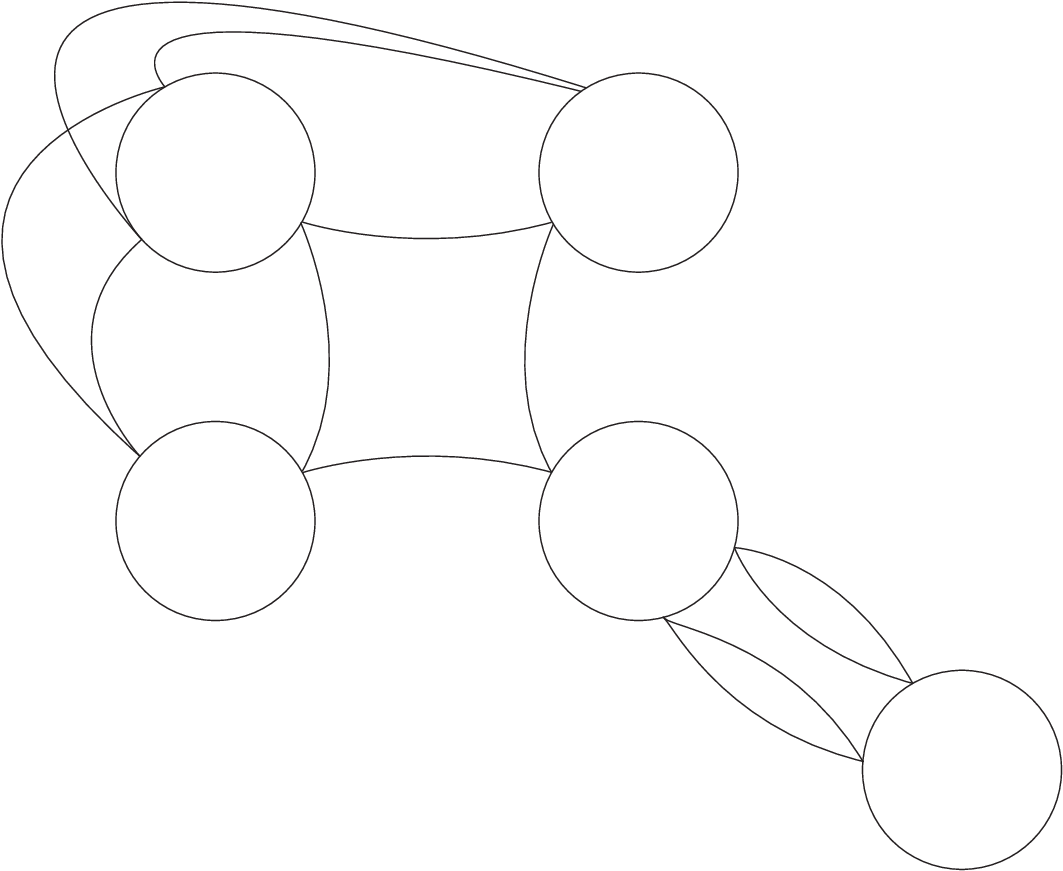}
    \caption{Partial dual $H_0^{\varnothing}=H_0$}
    \label{fig:bipartite_empty_dual}
  \end{subfigure}
\put(-395,136){\footnotesize$1^-$}
\put(-384,136){\footnotesize$1^+$}
\put(-376,125){\footnotesize$2^-$}
\put(-376,115){\footnotesize$2^+$}
\put(-383,107){\footnotesize$3^-$}
\put(-394,107){\footnotesize$3^+$}
\put(-401,116){\footnotesize$4^-$}
\put(-401,125){\footnotesize$4^+$}
\put(-310,136){\footnotesize$5^-$}
\put(-297,136){\footnotesize$5^+$}
\put(-292,126){\footnotesize$6^-$}
\put(-292,115){\footnotesize$6^+$}
\put(-297,107){\footnotesize$7^-$}
\put(-310,107){\footnotesize$7^+$}
\put(-319,116){\footnotesize$8^-$}
\put(-319,126){\footnotesize$8^+$}
\put(-378,34){\tiny$9^-$}
\put(-387,34){\tiny$9^+$}
\put(-380,50){\tiny$10^-$}
\put(-380,42){\tiny$10^+$}
\put(-319,42){\tiny$11^-$}
\put(-319,50){\tiny$11^+$}
\put(-303,35){\tiny$12^-$}
\put(-317,35){\tiny$12^+$}
\put(-85,110){\footnotesize1}
\put(-83,125){\footnotesize8}
\put(-140,110){\footnotesize4}
\put(-156,127){\footnotesize5}
\put(-156,105){\footnotesize7}
\put(-86,62){\footnotesize2}
\put(-71,52){\footnotesize11}
\put(-80,45){\footnotesize10}
\put(-140,62){\footnotesize3}
\put(-160,62){\footnotesize6}
\put(-30,22){\footnotesize12}
\put(-32,10){\footnotesize9}
  \caption{Bipartite case $E'=\varnothing$}
  \label{fig:bipartite_empty}
\end{figure}

\begin{figure}[h]
  \centering
  \begin{subfigure}[t]{0.351\textwidth}
    \centering
    \includegraphics[width=\textwidth]{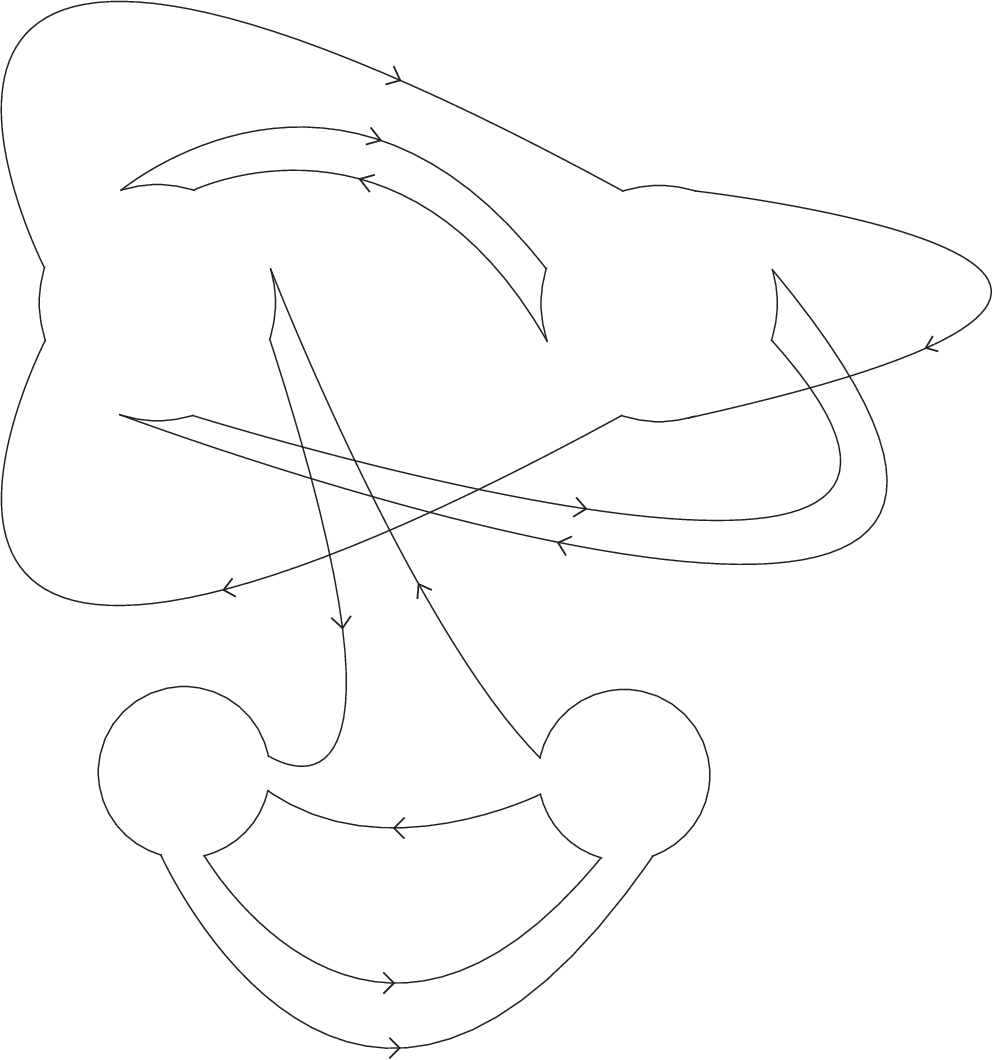}
    \caption{Medial map with all-crossing direction $\Phi_1$ ($C(\Phi_1)=\{C,D\}$)}
    \label{fig:bipartite_CD_medial}
  \end{subfigure}
  \hspace{2cm}
  \begin{subfigure}[t]{0.39\textwidth}
    \centering
    \includegraphics[width=\textwidth]{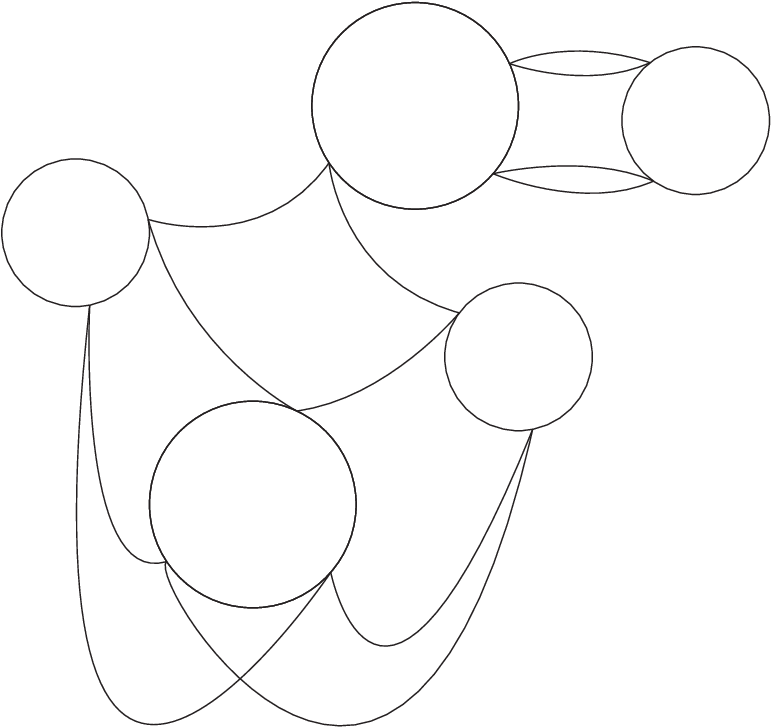}
    \caption{Partial dual $H_0^{\{C,D\}}$}
    \label{fig:bipartite_CD_dual}
  \end{subfigure}
\put(-395,136){\footnotesize$1^-$}
\put(-384,136){\footnotesize$1^+$}
\put(-376,125){\footnotesize$2^-$}
\put(-376,115){\footnotesize$2^+$}
\put(-383,107){\footnotesize$3^-$}
\put(-394,107){\footnotesize$3^+$}
\put(-401,116){\footnotesize$4^-$}
\put(-401,125){\footnotesize$4^+$}
\put(-310,136){\footnotesize$5^-$}
\put(-297,136){\footnotesize$5^+$}
\put(-292,126){\footnotesize$6^-$}
\put(-292,115){\footnotesize$6^+$}
\put(-297,107){\footnotesize$7^-$}
\put(-310,107){\footnotesize$7^+$}
\put(-319,116){\footnotesize$8^-$}
\put(-319,126){\footnotesize$8^+$}
\put(-378,34){\tiny$9^-$}
\put(-387,34){\tiny$9^+$}
\put(-380,50){\tiny$10^-$}
\put(-380,42){\tiny$10^+$}
\put(-319,42){\tiny$11^-$}
\put(-319,50){\tiny$11^+$}
\put(-303,35){\tiny$12^-$}
\put(-317,35){\tiny$12^+$}
\put(-155,113){\footnotesize1}
\put(-162,105){\footnotesize8}
\put(-120,65){\footnotesize4}
\put(-142,42){\footnotesize5}
\put(-110,35){\footnotesize7}
\put(-105,133){\footnotesize2}
\put(-77,152){\footnotesize11}
\put(-30,130){\footnotesize10}
\put(-65,95){\footnotesize3}
\put(-60,72){\footnotesize6}
\put(-30,152){\footnotesize12}
\put(-77,130){\footnotesize9}
  \caption{Bipartite case $E'=\{C,D\}$}
  \label{fig:bipartite_CD}
\end{figure}
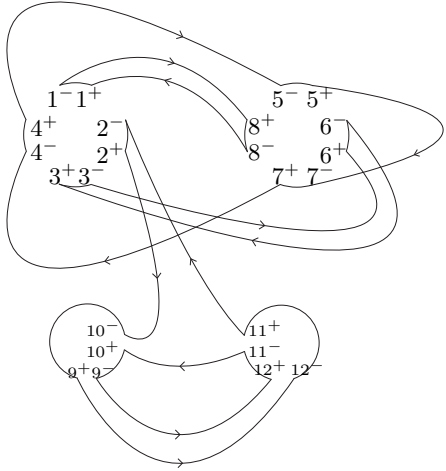
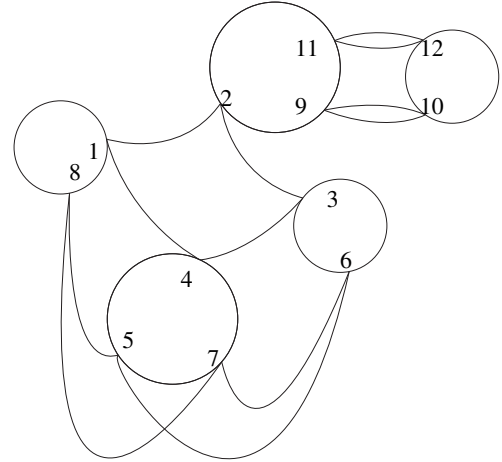

\end{document}